\documentclass[11pt]{article}
\usepackage{amsmath,amssymb,amsthm,color,epsfig,latexsym,wrapfig,verbatim,hyperref}
\usepackage{mathrsfs}
\usepackage{fullpage}
\newcommand {\apgt} {\ {\raise-.5ex\hbox{$\buildrel>\over\sim$}}\ }
\newcommand {\aplt} {\ {\raise-.5ex\hbox{$\buildrel<\over\sim$}}\ }

\newcommand{\F}{\mathcal F}

\newcommand{\length}{\mathsf{length}}
\newcommand{\diam}{\mathsf{diam}}

\newcommand{\f}{\varphi}
\newcommand{\Z}{\mathbb Z}
\newcommand{\e}{\varepsilon}
\newcommand{\E}{\mathbb E}
\newcommand{\R}{\mathbb R}

\newcommand{\Lip}{\mathrm{Lip}}

\newcommand{\lca}{ \mathrm{\mathbf{lca}}}

\newcommand{\dist}{\mathsf{dist}}

\newcommand\remove[1]{}

\newtheorem{theorem}{Theorem}[section]
\newtheorem{lemma}[theorem]{Lemma}

\newtheorem{proposition}[theorem]{Proposition}
\newtheorem{claim}[theorem]{Claim}

\newtheorem{corollary}[theorem]{Corollary}

\newtheorem{definition}[theorem]{Definition}

\newtheorem{remark}{Remark}[section]

{\makeatletter
 \gdef\xxxmark{%
   \expandafter\ifx\csname @mpargs\endcsname\relax 
     \expandafter\ifx\csname @captype\endcsname\relax 
       \marginpar{xxx}
     \else
       xxx 
     \fi
   \else
     xxx 
   \fi}
 \gdef\xxx{\@ifnextchar[\xxx@lab\xxx@nolab}
 \long\gdef\xxx@lab[#1]#2{{\bf [\xxxmark #2 ---{\sc #1}]}}
 \long\gdef\xxx@nolab#1{{\bf [\xxxmark #1]}}
}

\title{Trees and Markov convexity}

\author{James R. Lee\thanks{Computer Science and Engineering, University of Washington.  This research was
conducted while the author was at U. C. Berkeley and
the Institute for Advanced Study.  {\tt jrl@cs.washington.edu}} \\
\and Assaf Naor \thanks{Microsoft Research.  {\tt
anaor@microsoft.com}} \and Yuval Peres\thanks{Microsoft Research and
U. C. Berkeley. {\tt peres@stat.berkeley.edu}}}

\date{}
\begin{document}

\maketitle

\begin{abstract} We show that an infinite weighted tree
admits a bi-Lipschitz embedding into Hilbert space if and only if it
does not contain arbitrarily large complete binary trees with
uniformly bounded distortion. We also introduce a new metric
invariant called Markov convexity, and show how it can be used to
compute the Euclidean distortion of any metric tree up to universal
factors.
\end{abstract}

\section{Introduction}

Given two metric spaces $(X,d_X)$, $(Y,d_Y)$, and a mapping $f:X\to
Y$, we denote the Lipschitz constant of $f$ by $\|f\|_{\Lip}$. If
$f$ is injective then the (bi-Lipschitz) distortion of $f$ is
defined as $\dist(f)=\|f\|_{\Lip}\cdot \|f^{-1}\|_\Lip$. The
smallest distortion with which $X$ embeds into $Y$ is denoted
$c_Y(X)$, i.e. $c_Y(X)=\inf\{\dist(f):\ f:X\hookrightarrow  Y\}$.
When $Y=L_p$ for some $p\ge 1$ we use the shorter notation
$c_{p}(X)=c_{L_p}(X)$. The parameter $c_2(X)$ is known in the
literature as the Euclidean distortion of $X$.

The ubiquitous problem of embedding metric spaces into ``simpler"
spaces occurs in various aspects of functional analysis,
Riemannian geometry, group theory, and computer science. In most
cases low distortion embeddings are used to ``simplify" a
geometric object by representing it as a subset of a better
understood geometry. In other cases, embeddings are used to
characterize important invariants such as various notions of
dimensionality in metric spaces, and superreflexivity, type and
cotype of normed spaces. More recently, striking applications of
bi-Lipschitz embeddings were found in computer science, where the
information obtained from concrete geometric representations of
finite spaces is used to design efficient approximation algorithms
and
 data structures.


The present paper is devoted to the study of the Euclidean (and
$L_p$) distortion of trees. In what follows by a {\em metric tree}
we mean the shortest path metric induced on the vertices of a
weighted graph-theoretical tree $T=(V,E)$. In fact, all of our
results will hold true for arbitrary subsets of metric trees,
which are characterized among all metric spaces by the well known
four point condition: For every four points $x,y,u,v$ two of the
three numbers $\{d(x,y)+d(u,v),\ d(x,u)+d(y,v),\ d(x,v)+d(y,u)\}$
are the same, and that number is at least as large as the third
(see~\cite{Dress84}). But, because our statements are asymptotic
in nature, this does not increase the generality of our results,
since Gupta \cite{GuptaSteiner} proved that any finite subset of a
metric tree is bi-Lipschitz equivalent to a metric tree with
distortion at most $8$. The $\R$-tree corresponding to a tree $T$
is the one-dimensional simplicial complex induced by $T$, i.e. the
path metric obtained by replacing each edge in $T$ by an interval
whose length is the weight of the edge. The $\R$-tree
corresponding to $T$ will be denoted $T_\R$. In what follows, when
we refer to an $\R$-tree we mean an $\R$-tree corresponding to
some metric tree. We will see later that for every metric tree $T$
and every $p\ge 1$, $c_p(T)$ has the same order of magnitude as
$c_p(T_\R)$, so in most cases the distinction between metric trees
and $\R$-trees will not be important, though in a few instances we
will need to distinguish the two notions.

Let $B_k$ denote the complete binary tree of depth $k$ (with unit
edge weights). In a famous paper~\cite{Bou86-trees} Bourgain
proved that the Euclidean distortion of $B_k$ is
$\Theta\left(\sqrt{\log k}\right)$. Moreover, he showed that a
Banach space $Y$ is superreflexive (i.e. admits an equivalent
uniformly convex norm) if and only if $\lim_{k\to
\infty}c_Y(B_k)=\infty$. This remarkable characterization of a
linear property of Banach spaces in terms of their metric
structure sparked a considerable amount of work on problems of a
similar flavor (see the introduction of~\cite{MN05} for more
information on this topic). Among the corollaries of Bourgain's
work is the following dichotomy: For a Banach space $Y$ either
$c_Y(B_k)=1$ for all $k$, or there exists $\alpha>0$ such that
$c_Y(B_k)=\Omega\left(\left(\log k\right)^\alpha\right)$ (similar
phenomena are known to hold in a few other
cases---see~\cite{BMW86,MN05}). Moreover, Bourgain used his
theorem to solve a question posed by Gromov, showing that the
hyperbolic plane does not admit a bi-Lipschitz embedding into
Hilbert space. Similar applications of Bourgain's theorem to prove
that certain metric spaces do not embed into Hilbert space were
obtained by Benjamini and Schramm~\cite{BS97} in the case of
graphs with positive Cheeger constant, and by
Leuzinger~\cite{Leuz03} in the case of certain Tits buildings.

The bi-Lipschitz structure of trees has been studied extensively.
Trees are the ``building blocks" of hyperbolic geometry, and
embeddings of certain non-positively curved spaces into products of
trees are used in several contexts (see for
example~\cite{Dra03,BS05,NPSS04,LS05}).  Similar results (known as
``probabilistic embeddings into trees") are a powerful tool in
computer science (see for example~\cite{Bar98,FRT03}). We refer
to~\cite{Dress84,Mat90,JLPS02} for other results on the Lipschitz
structure of trees. In spite of these applications, and the  vast
amount of work on trees in the Lipschitz category, the following
natural problem remained open: When does an infinite metric tree
embed with finite distortion into Hilbert space? One of the main
results of this paper is the following answer to this question.

\begin{theorem}\label{thm:local} Let  $T=(V,E)$ be an infinite metric tree. Then the
following conditions are equivalent.
\begin{enumerate}
\item $c_2(T)=\infty$.
 \item $\sup_{k\in \mathbb N} c_T(B_k)<\infty$.
\item For every $k\in  \mathbb N$, $c_T(B_k)=1$.
\end{enumerate}
\end{theorem}

In other words, a metric tree admits a bi-Lipschitz embedding into
Hilbert space if and only if it does not bi-Lipschitzly contain
 arbitrarily large complete binary trees. Thus there
is a unique obstruction to a tree being non-Euclidean. Similar
``unique obstruction" results are known only in very few cases: As
we mentioned above, Bourgain~\cite{Bou86-trees} proved that
complete binary trees are the unique obstruction to a Banach space
being superreflexive; Bourgain, Milman and Wolfson~\cite{BMW86}
showed that Hamming cubes are the unique obstruction to a metric
space having non-trivial type; Mendel and Naor~\cite{MN05} showed
that $\ell_\infty^n$ integer grids are the unique obstruction to a
metric space having finite cotype;  Thomassen~\cite{Tom92} proved
that certain transient graphs must contain transient trees, and
Benjamini and Schramm~\cite{BS97} proved that a graph with
positive Cheeger constant must contain a tree with positive
Cheeger constant. Another result in the spirit of
Theorem~\ref{thm:local} is the tree Szemeredi theorem of
Furstenberg and Weiss~\cite{FW03}: A subset of positive density in
the infinite complete binary tree must contain arbitrarily large
copies of complete binary trees.

It is not surprising that the Theorem~\ref{thm:local} is a ``local"
result,  in the sense that it deals with finite subsets of the
metric tree $T$. Indeed, it is well known that a metric space embeds
into Hilbert space if and only if all of its finite subsets do. It
is thus natural to expect characterizations in the spirit of
Theorem~\ref{thm:local} to be local. Let us say that a metric space
$X$ is {\em finitely representable in a metric space $Y$} if there exists
a constant $D \geq 1$ such that for every finite subset $F\subseteq
X$ we have $c_Y(F)\le D$ (this is an obvious adaptation of standard
terminology from Banach space theory).  Thus, denoting by $B_\infty$
the infinite unweighted complete binary tree,
Theorem~\ref{thm:local} can be rephrased as follows: A metric tree
$T$ admits a bi-Lipschitz embedding into Hilbert space if and only if
$B_\infty$ is not finitely representable in $T$.  The following
section contains optimal quantitative versions of this result, and
explains the ingredients of its proof.

\subsection{Markov convexity and quantitative bounds}

Quantitative bounds on the Euclidean distortion of trees were
obtained in~\cite{LMS98,Mat99,LS03,GKL03}. In particular,
Matou\v{s}ek proved in~\cite{Mat99} that for any $n$-point metric
tree $T$ we have $c_2(T)=O\left(\sqrt{\log \log n}\right)$. This
result cannot be improved due to Bourgain's lower bound for the
complete binary tree.  Gupta, Krauthgamer and Lee \cite{GKL03}
obtained upper bounds on the Euclidean distortion of trees in
terms of their doubling constant;
in particular, they showed that every doubling tree
admits a bi-Lipschitz embedding into a finite-dimensional
Euclidean space.
We present a new
 simpler proof of this fact in Section \ref{sec:doubling},
where we also recall the definition of the doubling constant.

We shall now state an optimal quantitative version of
Theorem~\ref{thm:local}. Given a metric space $(X,d_X)$, $k\in
\mathbb N$ and $c>1$, we denote
$$
\mathscr{B}_X(c)=\max\left\{k\in \mathbb N:\ c_X(B_k)<c\right\}.
$$
In what follows we write $A \aplt B$ to mean $A = O(B)$. If
$A\aplt B$ and $B\aplt A$ then we write $A\approx B$.
\begin{theorem}\label{thm:quantitative-local}
Let $T$ be an arbitrary metric tree. Then for every $c>1$,
$$
\frac{1}{c}\sqrt{\log \mathscr B_T(c)}\aplt c_2(T)\aplt
\sqrt{\frac{c}{c-1}\cdot \mathscr{B}_T(c)}.
$$
\end{theorem}
The lower bound in Theorem~\ref{thm:quantitative-local} is simply
Bourgain's lower bound, and is therefore optimal. Somewhat
surprisingly, the upper bound in
Theorem~\ref{thm:quantitative-local} cannot be improved. The
construction of a family of trees exhibiting this, which we call the Cantor
trees, is presented in Section~\ref{sec:cantor}.

It follows that in order to obtain tight bounds on the Euclidean
distortion of a given metric tree $T$ we need an invariant which is
more refined than the size of the biggest binary tree contained in
$T$. This is achieved via the following definition. Let
$\{X_t\}_{t=0}^\infty$ be a Markov chain on a state space $\Omega$.
Given an integer $k\ge 0$ we denote by $\{\widetilde
X_t(k)\}_{t=0}^\infty$ the process which equals $X_t$ for time $t\le
k$, and evolves independently (with respect to the same transition
probabilities) for time $t>k$. Observe that for $k < 0$, $\widetilde
X_t(k)$ and $X_t$ evolve independently for all $t \geq 0$.

\begin{definition}\label{def:Markov}
Let $(X,d)$ be a metric space and $p> 0$. We shall say that $(X,d)$
is {\em Markov $p$-convex} with constant $\Pi$ if for every Markov
chain $\{X_t\}_{t=0}^\infty$ on a state space $\Omega$, and every
$f:\Omega\to X$, we have for every $m\in \mathbb N$,
\begin{eqnarray}\label{eq:Markov convexity}
\sum_{k=0}^m\sum_{t=1}^{2^{m}}\frac{\E\left[d\left(f\left(X_{t}\right),f(\widetilde
X_{t}(t-2^k))\right)^p\right]}{2^{kp}} \le \Pi^p \sum_{t=1}^{2^m}\E
[d(f(X_t),f(X_{t-1}))^p].
\end{eqnarray}
The least constant $\Pi$ above is called the Markov $p$-convexity
constant of $X$, and is denoted $\Pi_p(X)$. We shall say that $X$ is
Markov $p$-convex if $\Pi_p(X)<\infty$.
\end{definition}

To understand this notion, recall that the chains $X_{t}$ and
$\widetilde X_{t}(t-2^k)$ run together for the first $t-2^k$ steps,
and then evolve independently for the remaining $2^{k}$ steps.  Thus
the left hand side in~\eqref{eq:Markov convexity} is measuring the
sum, over many ``dyadic scales'' $k \in\{ 0, 1, 2, \ldots\}$ of the
average of the $p$th power of the normalized ``drift'' of the chain
in $X$ with respect to scale $k$.  We will say that $X$ has {\em
non-trivial Markov convexity} if $X$ is Markov $p$-convex for some
$p < \infty$. We note that $L_2$ is Markov $2$-convex. More
generally, the name comes from the fact that if $X$ is a Banach
space which admits an equivalent uniformly convex norm whose modulus
of convexity is of power type $p$, then $X$ is also Markov
$p$-convex. These results are proved in Section~\ref{sec:banach}.

In Bourgain's paper \cite{Bou86-trees} there is an implicit
``non-linear'' notion of uniform convexity related to the presence
of complete binary trees. For the results in this paper, we
require the above ``Markov variant,'' analogous to Ball's notion
of Markov type~\cite{Ball92}. The search for Poincar\'e-type
inequalities on metric spaces which are analogs of classical
Banach space invariants have been fruitfully investigated by
several authors---we refer to the
papers~\cite{Enflo78,Gromov83,BMW86,Pisier86,Ball92,NS02,NPSS04,MN05}
for a discussion of this research direction, to which
Definition~\ref{def:Markov} belongs. The following theorem shows
that Markov convexity determines the Euclidean distortion of a
tree, up to universal factors.

\begin{theorem}\label{thm:formula}
Let $T$ be a metric tree. Then $c_2(T)\approx \Pi_2(T_\R)$.
\end{theorem}
Recall that $T_\R$ denotes the $\R$-tree corresponding to $T$. See
Remark \ref{rem:rtree} for a discussion of why we have to pass to
$\mathbb R$-trees in Theorem~\ref{thm:formula}.

We also obtain a combinatorial way to compute the Euclidean
distortion of any tree. Let $T=(V,E)$ be a metric tree, and let
$\chi:E\to \Z$ be an edge coloring. We call $\chi$ a monotone
coloring if all of its color classes are paths contained in a
root-leaf path (such paths are called monotone paths in what
follows). For $\delta\in (0,1)$, the coloring $\chi$ is called
$\delta$-strong if it is monotone and for every $u,v\in V$ at
least half of the length of the path joining $u$ and $v$ can be
covered by color classes of length at least $\delta d_T(u,v)$. We
define $\delta^*(T)$ to be the largest $\delta$ for which $T$
admits a $\delta$-strong coloring. The following theorem shows
that $\delta^*(T)$ determines the Euclidean distortion of $T$.

\begin{theorem}\label{thm:star}
Let $T$ be a metric tree. Then
$$
c_2(T)\approx \sqrt{1+\log\left(\frac{1}{\delta^*(T)}\right) }.
$$
\end{theorem}
The upper bound on $c_2(T)$ in Theorem~\ref{thm:star} continues a
theme that also appeared in~\cite{LMS98,Mat99,GKL03}: Certain edge
colorings can be used to construct embeddings into $L_2$.
Specifically, our proof draws on ideas from Matou\v{s}ek's
embedding~\cite{Mat99}. But, Matou\v{s}ek's argument requires
colorings with a small number of colors, the existence of which
depends only on the topology of $T$ and does not take into account
the edge lengths. Our argument for the upper bound, which is
contained in Theorem~\ref{thm:matousek}, builds on Matou\v{s}ek's
proof while taking the metric into consideration, and is therefore
more involved.

The lower bound on $c_2(T)$ in Theorem~\ref{thm:star} goes through
Theorem~\ref{thm:formula}. We construct a special coloring of $T$,
and show that if the coloring is not $\delta$-strong, then we can
construct a Markov chain on $T$ which wanders too quickly for $T$
to have a small Markov 2-convexity constant. This is done by
locating a special type of subtree of $T$, which we call a weak
prototype---see Section~\ref{sec:def weak} for the definition,
where it is shown that weak prototypes cannot have good Markov
convexity properties. This ``reconstruction paradigm'' is inspired
by a result of~\cite{GKL03} which shows that if a certain
procedure fails to produce a good coloring, then the tree under
consideration must have a large doubling constant. Our approach is
able to pick out significantly more delicate sub-structures (e.g.
embedded complete binary trees or the aforementioned ``weak
prototypes''). A key difficulty that arises in our setting
involves choosing the ``scale" at which the required weak
prototype embeds into $T$. This ``scale selection" argument is a
central part of our proof of Theorem~\ref{thm:quantitative-local},
Theorem~\ref{thm:formula}, and Theorem~\ref{thm:star}---we refer
to Section~\ref{sec:binary} and Section~\ref{sec:dist formula} for
the details.

We remark that all of our results can be applied to compute the
$L_p$ distortion of trees. Namely, we show that for every $p,c>1$
and every metric tree $T$,
\begin{eqnarray}\label{eq:Lpformula1}
\frac{1}{c}\left(\log \mathscr
B_T(c)\right)^{\min\left\{\frac{1}{p},\frac12\right\}}\aplt
c_p(T)\aplt \left(\frac{c}{c-1}\cdot
\mathscr{B}_T(c)\right)^{\min\left\{\frac{1}{p},\frac12\right\}}.
\end{eqnarray}
and
\begin{eqnarray}\label{eq:Lpformula2}
c_p(T)\approx \Pi_{\max\{p,2\}}(T_\R)\approx
\left[\log\left(\frac{2}{\delta^*(T)}\right)\right]^{\min\left\{\frac{1}{p},\frac12\right\}},
\end{eqnarray}
where the implied constants may depend only on $p$; see
Theorem~\ref{thm:Markov main}.

\medskip

The use of Markov convexity as a metric invariant, and thus a tool
for proving distortion lower bounds, is not limited to the case of
trees. In Section~\ref{sec:lower distortion} we investigate
classes of spaces which can be shown not to embed into $L_2$ by
analyzing their Markov convexity. In particular, we prove a lower
bound on the Euclidean distortion of balls of finitely generated
groups (equipped with the word metric) which admit a bounded
non-constant harmonic function. We also bound from below the
Euclidean distortion of the lamplighter group over the cycle (see
Section~\ref{sec:lower distortion} for the definition). In a
future paper, which will be devoted to embeddings of the
lamplighter group, we use the methods of~\cite{NPSS04} to show
that this group has Markov type $2$ in the sense of
Ball~\cite{Ball92}. Thus, Markov convexity is the only known
invariant which demonstrates that this group does not well-embed
into Hilbert space.

Our results, specifically Theorem~\ref{thm:star}, have algorithmic
implications. Given an $n$-point metric space $X$, the problem of
efficiently computing its distortion in a class of metric spaces
up to a small factor has attracted a lot of attention in recent
years, and is known as the ``relative embedding" problem. We refer
to~\cite{BCIS05} and the references therein for a discussion of
this topic, and also for some hardness results. The Euclidean
distortion of an $n$-point metric space can be computed in
polynomial time, since this problem can be cast as a semidefinite
program~\cite{LLR95}. Hence Theorem~\ref{thm:star} yields a
polynomial time algorithm for estimating the parameter $\log\left(
\frac{1}{\delta^*(T)}\right)$ up to a constant factor for any tree
$T$. In conjunction with~\eqref{eq:Lpformula2}, this gives a
polynomial time algorithm which computes the $L_p$ distortion of
any tree up to a universal constant factor. Note that it not known
whether the $L_p$ distortion of a general finite metric space can
be approximated efficiently.

\subsection{Some open problems}

We end this introduction by stating some interesting open problems
that arise from our work.

\bigskip

\noindent{\bf Problem 1.} In Section~\ref{sec:banach} we show that
every $p$-uniformly convex Banach space is Markov $p$-convex. We
also show that if $X$ is a Banach lattice which is Markov
$p$-convex then it is also $q$-uniformly convex for every $q>p$.
The relation between Markov $p$-convexity and uniform
$p$-convexity for general Banach spaces is unclear.

\bigskip

\noindent{\bf Problem 2.} One corollary of our results is that if
a metric tree is not Markov $p$-convex for any $p<\infty$ then it
contains arbitrarily large complete binary trees with distortion
arbitrarily close to $1$. It is possible that this holds true for
arbitrary metric spaces, and not just metric trees. If this is the
case, then it would correspond to known results in Banach space
theory, and would complement the existing theory of metric type
and cotype.

\bigskip

\noindent{\bf Problem 3.} It would be interesting to investigate
other ``unique obstruction" results of the type described here. In
particular, can one classify the obstructions to a planar graph
being embeddable in $L_2$?  Another interesting generalization
would be to classify the subsets of $\mathbb H^2$---the hyperbolic
plane---which embed into $L_2$; it seems plausible that complete
binary trees are the only obstruction in this case, just as for
tree metrics. In a similar vein, it might be the case that the
only subsets of a product of trees which do not admit a
bi-Lipschitz embedding into $L_2$ are those that contain
arbitrarily large bi-Lipschitz copies of complete binary trees. If
true, then in combination with the result of~\cite{BS05}, this
would imply the same result for the hyperbolic plane.

\bigskip

\noindent{\bf Problem 4.} In Section~\ref{sec:lower distortion} we
give lower bounds on the Euclidean distortion of the lamplighter
group over the $n$-cycle. We do not know what is the correct
asymptotic behavior of this distortion. It is also unknown whether
or not these groups embed into $L_1$ with uniformly bounded
distortion.

\bigskip
\begin{figure}[h]
\begin{center}\includegraphics[scale=0.80]{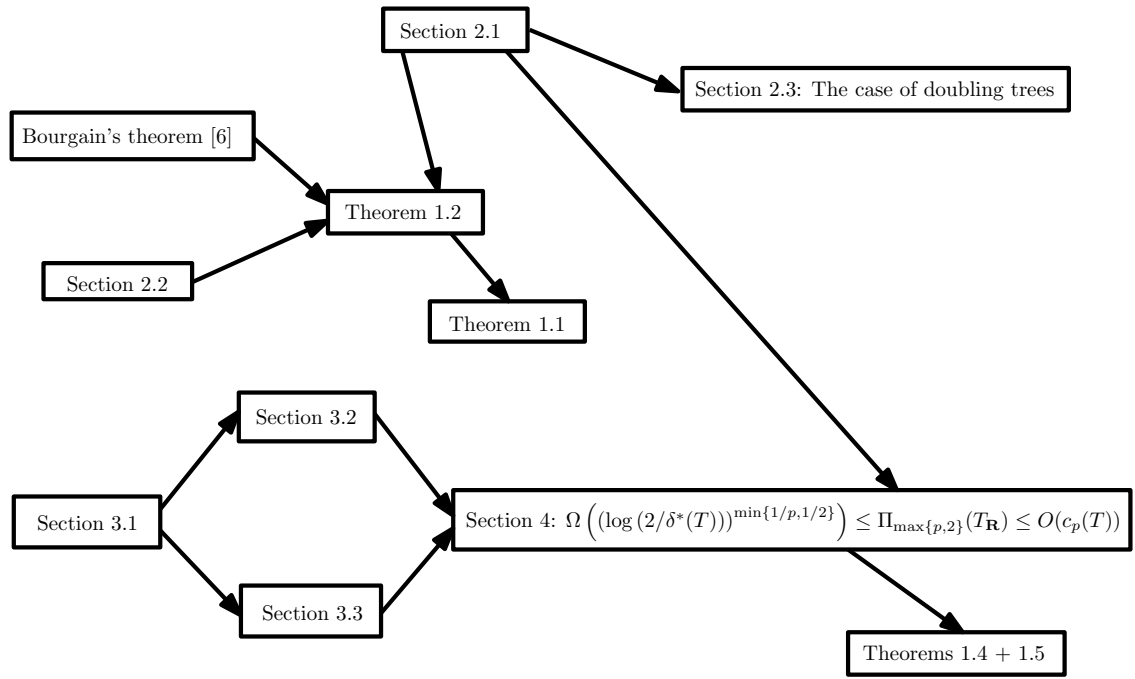}
\end{center}
 \caption{{{\em  A schematic description of the implications between the sections in this paper. }}} \label{fig:midpoints}
\end{figure}

\bigskip

\section{Distortion bounds via the containment of binary trees}

The purpose of this section is to prove the following theorem
which, when combined with Bourgain's lower bound for binary
trees~\cite{Bou86-trees}, implies Theorem~\ref{thm:local} and
Theorem~\ref{thm:quantitative-local}.
\begin{theorem}\label{thm:upper-binary}
Let $T$ be an arbitrary metric tree and $p\ge 1$. Then for every
$c>1$,
$$
c_p(T)\le 130\left(\frac{c}{c-1}\cdot
\mathscr{B}_T(c)\right)^{\min\left\{\frac{1}{p},\frac12\right\}}.
$$
\end{theorem}
In Section~\ref{sec:cantor} we will show that the asymptotic
dependence
on $\mathscr{B}_T(c)$ in the upper bound of
Theorem~\ref{thm:upper-binary} cannot be improved.

\subsection{Coloring based upper bounds}\label{section:coloring}

We begin with some definitions and notation. Let $T = (V, E)$ be a
{\em finite} graph-theoretic tree with positive edge lengths $\ell
: E \to (0,\infty)$, and let $d_T$ be the induced path metric on
$V$. We also fix some arbitrary root $r \in T$. A {\em monotone
path} in $T$ is a connected subset (also called a {\em segment} in
what follows) of some root-leaf path. By an {\em edge-coloring of
$T$}, we mean a map $\chi : E \to \mathbb Z$. We say that the
coloring is {\em monotone} if for every $m\in \Z$ the color class
$\chi^{-1}(m)$ is a monotone path. For $u,v\in V$ we let
$P(u,v)\subseteq E$ denote the unique path from $u$ to $v$, and
set $P(v)=P(v,r)$. Given an edge coloring $\chi:E\to \Z$, $k \in
\chi(E)$, and $u,v\in V$, we write
$$\ell_k^\chi(u,v):= \sum_{\substack{e\in P(u,v)\\\chi(e)=k}}\ell(e).$$
We also set $\ell_k^\chi(v)=\ell_k^\chi(v,r)$. Finally, given
$u,v\in V$ we let $\lca(u,v)$ denote the {\em least common
ancestor} of $u$ and $v$ in $T$.

\begin{definition}[$\e$-good coloring]\label{def:epsilon strong}
 We say that a coloring $\chi:E\to \Z$ is $\varepsilon$-good if it is monotone, and for every
$u,v \in T$, the unique path from $u$ to $v$ contains a
monochromatic segment of length at least $\varepsilon \cdot d_T(u,v)$.
We define $\e^*(T)$ to be the largest $\e$ for which $T$ admits an
$\e$-good coloring.
\end{definition}

The following simple lemma will not be used in the proof of
Theorem~\ref{thm:upper-binary}, but we include it since it
illustrates the relation between colorings and embeddings,
and it will be used eventually in Section \ref{sec:doubling}.

\begin{lemma}\label{claim:color}
For every weighted tree $T$ and $p\ge 1$,
$$
c_p(T)\le \frac{2^{1/p}}{\e^*(T)}.
$$
\end{lemma}

\begin{proof} Fix $\e<\e^*(T)$ and let $\chi:E\to \Z$ be an $\e$-good coloring. Let $\{e_k\}_{k\in \Z}$ be the standard basis of
$\ell_p=\ell_p(\Z)$. Define $f:V\to \ell_p$ by
$$
f(v)=\sum_{k\in \Z}\ell_k^\chi(v)e_k.
$$
(Recall that $\ell_k(v)$ is the distance that the segment colored
$k$ contributes to the path joining $v$ to the root).

Fix $u,v\in V$ and write $w=\lca(u,v)$. The fact that the coloring
$\chi$ is monotone implies that $\chi(P(u,w))\cap
\chi(P(v,w))=\emptyset$. Thus
$$
d_T(u,v)=\sum_{e\in P(u,w)} \ell(e)+\sum_{e\in P(v,w)}
\ell(e)=\sum_{k\in \Z}
|\ell_k^\chi(u)-\ell_k^\chi(v)|=\|f(u)-f(v)\|_1\ge
\|f(u)-f(v)\|_p.
$$
On the other hand, since $\chi$ is $\e$-good, there are $a,b\in
\Z$ such that $\ell_a^\chi(u,w)\ge \e d_T(u,w)$ and
$\ell_b^\chi(v,w)\ge \e d_T(v,w)$. It follows that
\begin{equation}\label{eq:goodcolor}
\|f(u)-f(v)\|_p\ge
\left([\ell_a^\chi(u,w)]^p+[\ell_b^\chi(v,w)]^p\right)^{1/p}\ge
\frac{\e}{2^{1/p}} [d_T(u,w)+d_T(v,w)]= \frac{\e}{2^{1/p}}d_T(u,v).
\end{equation}
\end{proof}

To get tighter control on the Euclidean distortion of trees we
introduce the notion of $\delta$-strong colorings.

\begin{definition}[$\delta$-strong coloring]\label{def:delta
strong} We say that a coloring $\chi:E\to \Z$ is $\delta$-strong
if it is monotone, and for every $u,v \in V$
$$
\sum_{k\in \Z} \ell^\chi_k(u,v)\cdot {\bf
1}_{\{\ell^\chi_k(u,v)\ge \delta d_T(u,v)\}}\ge \frac12 d_T(u,v).
$$
In words, we demand that at least half of the shortest path joining
$u$ and $v$ is covered by color classes of length at least $\delta
d_T(u,v)$. As before, we define $\delta^*(T)$ to be the largest
$\delta$ for which $T$ admits an $\delta$-strong coloring.
\end{definition}

The notions of $\delta$-strong colorings and $\e$-good colorings are
related via the following simple lemma.

\begin{lemma}\label{lem:relation}
Every weighted tree $T$ satisfies $ \delta^*(T)\ge 2^{-3/\e^*(T)}$.
\end{lemma}

\begin{proof} Let $\chi$ be an $\e$-good coloring of $T$. We will
prove that it is also $2^{-3/\e}$-strong. In fact, we shall show
that for every $\alpha\in (0,1]$ and $u,v\in V$, the total length
of the monochromatic segments of length at least $\alpha d_T(u,v)$
on the path $P(u,v)$ satisfies
\begin{eqnarray}\label{eq:geometric}
\sum_{k\in \Z} \ell^\chi_k(u,v)\cdot {\bf 1}_{\{\ell^\chi_k(u,v)\ge
\alpha d_T(u,v)\}}\ge
\left(1-\left(\frac{\alpha}{\e}\right)^{\e/2}\right)d_T(u,v).
\end{eqnarray}
Choosing $\alpha=2^{-3/\e}$ in~\eqref{eq:geometric}, and using the
fact that $2^{1/\e}\ge \frac{1}{\e}$, we deduce that $\chi$ is
$2^{-3/\e}$-strong. The proof of~\eqref{eq:geometric} is by
induction on $d_T(u,v)$. If $d_T(u,v)$ is minimal then $P(u,v)$ is
an edge, and hence monochromatic, so that the assertion holds
trivially. In general, since the coloring $\chi$ is $\e$-good,
there are two vertices on the path $P(u,v)$ such that the segment
$P(a,b)$ is monochromatic and $d_T(a,b)\ge \e d_T(u,v)$. Without
loss of generality we assume that $d_T(a,u)<d_T(b,u)$. If
$\e<\alpha$ then there is nothing to prove, so assume that $\e\ge
\alpha$. Denoting $A=d_T(u,a)$, $B=d_T(b,v)$ and $D=d_T(u,v)$, we
apply the inductive hypothesis to the paths joining $u$ and $a$
and $b$ and $v$, to get that
\begin{eqnarray}
\sum_{k\in \Z} \ell^\chi_k(u,v)\cdot {\bf 1}_{\{\ell^\chi_k(u,v)\ge
\alpha d_T(u,v)\}}&\ge& d_T(a,b)+
\left(1-\left(\frac{D\alpha}{A\e}\right)^{\e/2}\right)A+
\left(1-\left(\frac{D\alpha}{B\e}\right)^{\e/2}\right)B\nonumber\\
&=&\nonumber
D-\left(\frac{D\alpha}{\e}\right)^{\e/2}\left(A^{1-\e/2}+B^{1-\e/2}\right)\\
&\ge&\label{eq:concavity}
D-2\left(\frac{D\alpha}{\e}\right)^{\e/2}\left(\frac{A+B}{2}\right)^{1-\e/2}\\
&\ge&\label{eq:use-big}
D-2\left(\frac{D\alpha}{\e}\right)^{\e/2}\left(\frac{(1-\e)D}{2}\right)^{1-\e/2}\\
&\ge& \label{eq:graph}
\left(1-\left(\frac{\alpha}{\e}\right)^{\e/2}\right)D,
\end{eqnarray}
Where in~\eqref{eq:concavity} we used the concavity of the function
$t\mapsto t^{1-\e/2}$, in~\eqref{eq:use-big} we used the fact that
$D=A+B+d_T(a,b)\ge A+B+\e D$, and in~\eqref{eq:graph} we used the
elementary inequality $2^{\e/2}(1-\e)^{1-\e/2}\le 1$, which is valid
for every $\e\in [0,1]$.
\end{proof}

In \cite{Mat99}, Matou{\v{s}}ek proves that if $\chi$ is a
monotone edge-coloring of $T$ such that every root-leaf path
contains at most $M$ distinct color classes, then $c_p(T) \leq
O\left((\log M)^{\min\{\frac1p, \frac12\}}\right)$. Clearly any
such coloring is $1/(2M)$-strong.  The next theorem
generalizes Matou{\v{s}}ek's
result along these lines.

We suggest that the
reader skip this rather technical proof upon a first reading.
In particular, the much simpler Lemma~\ref{claim:color}
suffices for the proof of Theorem \ref{thm:local},
although it does not give the optimal quantitative bounds
of Theorem \ref{thm:quantitative-local}.

\begin{theorem}\label{thm:matousek}
For every weighted tree $T=(V,E)$ and $p\ge 1$,
$$
c_p(T)\le
4\left[\log\left(\frac{2}{\delta^*(T)}\right)\right]^{\min\left\{\frac{1}{p},\frac12\right\}}.
$$
\end{theorem}

\begin{proof} We may assume that $p\in [2,\infty)$, since if $p\in
[1,2)$ the required result follows by embedding $T$ into
$\ell_2\subseteq L_p$. Fix $\delta<\min\{\delta^*(T),1/2\}$ and
let $\chi:E\to \Z$ be a $\delta$-strong coloring. Let
$\{e_k\}_{k\in \Z}$ be as in Lemma~\ref{claim:color}. For $v\in V$
we denote by $(k_1(v),\ldots,k_{m_v}(v))$ the sequence of color
classes encountered on the path from the root to $v$. We shall
denote by $d_j(v)$ the distance that the color class $k_j(v)$
contributes to the path from the root to $v$, i.e.
$$
d_j(v)=\sum_{\substack{e\in P(v)\\ \chi(e)=k_j(v)}}\ell(e).
$$
Now let
$$
s_i(v)=\sum_{j=i}^{m_v}\max\left\{0,d_j(v)-\frac{\delta}{2}\sum_{h=i}^{j}d_h(v)\right\},
$$
and define $f:V\to \ell_p(\Z)$ by
$$
f(v)=\sum_{i=1}^{m_v} [d_i(v)]^{1/p}[s_i(v)]^{(p-1)/p} e_{k_i(v)}.
$$

We will break the proof the proof of the fact that $f$ satisfies
the required distortion bound into several claims.

\begin{claim}\label{claim:use delta strong} For all $v\in V$ and $j\in \{1,\ldots, m_v\}$,
$$
s_i(v)\ge \frac14 \sum_{j=i}^{m_v} d_j(v).
$$
\end{claim}
\begin{proof}
This is where the fact that $\chi$ is a $\delta$-strong coloring
comes in. Indeed,
\begin{eqnarray*}
s_i(v)=\sum_{j=i}^{m_v}\max\left\{0,d_j(v)-\frac{\delta}{2}\sum_{h=i}^{j}d_h(v)\right\}\ge
\sum_{\substack{j\in \{i,\ldots,m_v\}\\ d_j(v)\ge \delta
\sum_{h=i}^{m_v}d_h(v)}}\frac{d_j(v)}{2}\ge \frac14
\sum_{j=i}^{m_v}d_j(v).
\end{eqnarray*}
\end{proof}

\begin{claim}\label{claim:flip}
$\|f\|_{\Lip}\le [5\log(3/\delta)]^{1/p}$.
\end{claim}
\begin{proof}
We need to show that for every edge $(u,v)\in E$,
$\|f(u)-f(v)\|_p\le 10[\log(1/\delta)]^{1/p}$. Assume that $v$ is
further than $u$ from the root of $T$. In this case
$k_1(u)=k_1(v),\ldots, k_{m_u}(u)=k_{m_u}(v)$ and $m_v\in
\{m_u,m_u+1\}$. If $m_v=m_u+1$ then we denote for the sake of
simplicity $d_{m_v}(u)=s_{m_v}(u)=0$. With this notation we have
that
\begin{eqnarray*}
\|f(u)-f(v)\|_p^p&=&
\sum_{i=1}^{m_v}\left|[d_i(u)]^{1/p}[s_i(u)]^{(p-1)/p}-[d_i(v)]^{1/p}[s_i(v)]^{(p-1)/p}\right|^p\\
&=& \sum_{i=1}^{m_v-1}
d_i(v)\left|[s_i(u)]^{(p-1)/p}-[s_i(v)]^{(p-1)/p}\right|^p+\\&\phantom{\le}&
\left|[d_{m_v}(u)]^{1/p}[s_{m_v}(u)]^{(p-1)/p}-[d_{m_v}(v)]^{1/p}[s_{m_v}(v)]^{(p-1)/p}
\right|^p.
\end{eqnarray*}
Note that by our definitions, $s_{m_v}(u)=d_{m_v}(u)$ and
$s_{m_v}(v)=d_{m_v}(v)$. Thus
\begin{eqnarray}\label{eq:first bound}
\|f(u)-f(v)\|_p^p&=&\sum_{i=1}^{m_v-1}\nonumber
d_i(v)\left|[s_i(u)]^{(p-1)/p}-[s_i(v)]^{(p-1)/p}\right|^p+\left|d_{m_v}(u)-d_{m_v}(v)\right|^p\\&\le&
\sum_{i=1}^{m_v-1}
d_i(v)\left|[s_i(u)]^{(p-1)/p}-[s_i(v)]^{(p-1)/p}\right|^p+[d_T(u,v)]^p.
\end{eqnarray}

Observe that for all $i\in \{1,\ldots,m_v-1\}$, $s_i(v)\ge
s_i(u)$. Thus
\begin{eqnarray}\label{eq:numerical}
\left|[s_i(u)]^{(p-1)/p}-[s_i(v)]^{(p-1)/p}\right|\le
\frac{|s_i(u)-s_i(v)|}{[s_i(v)]^{1/p}},
\end{eqnarray}
where we used the elementary inequality $y^\alpha-x^\alpha\le
\frac{y-x}{y^{1-\alpha}}$, which is valid for all $y\ge x>0$ and
$\alpha \in (0,1)$.

Observe that for every $i\le m_v-1$,
\begin{eqnarray}\label{eq:trunc-lip}
s_i(v)-s_i(u)=
\max\left\{0,d_{m_v}(v)-\frac{\delta}{2}\sum_{h=i}^{{m_v}}d_h(v)\right\}-
\max\left\{0,d_{m_v}(u)-\frac{\delta}{2}\sum_{h=i}^{{m_v}}d_h(u)\right\}\le
d_T(u,v).
\end{eqnarray}
\remove{ Indeed, to check the validity of~\eqref{eq:trunc-lip}
assume first of all that $m_u=m_v$. In this case let
$a=d_{m_v}(u)$, $b=\sum_{h=i}^{m_v} d_h(u)$ and
$c=d_{m_v}(v)-d_{m_v}(u)=d_T(u,v)$. With this
notation~\eqref{eq:trunc-lip} becomes,
$$
s_i(v)-s_i(u)=\max\left\{0,a+c-\frac{\delta}{2}(b+c)\right\}-\max\left\{0,a-\frac{\delta}{2}b\right\}.
$$
This is at most $c=d_T(u,v)$ unless $a+c\ge \delta(b+c)$ and
$a<\delta b$, in which case it equals $a+c$.}

Thus, combining~\eqref{eq:numerical} and~\eqref{eq:trunc-lip} we
see that
\begin{eqnarray}\label{eq:first sum}
\sum_{i=1}^{m_v-1}
d_i(v)\left|[s_i(u)]^{(p-1)/p}-[s_i(v)]^{(p-1)/p}\right|^p&\le&\nonumber
\sum_{i=1}^{m_v-1} d_i(v)\cdot
\frac{|s_i(u)-s_i(v)|^p}{s_i(v)}\\&\le&\nonumber
[d_T(u,v)]^p\cdot\sum_{\substack{i\in\{1,\ldots,m_v-1\}\\s_i(u)\neq
s_i(v)} }\frac{d_i(v)}{s_i(v)}\\
&\le&
4[d_T(u,v)]^p\cdot\sum_{\substack{i\in\{1,\ldots,m_v-1\}\\s_i(u)\neq
s_i(v)} }\frac{d_i(v)}{\sum_{j=i}^{m_v} d_j(v)},
\end{eqnarray}
where in the last line we used Claim~\ref{claim:use delta strong}.

Observe that for every $x_1,\ldots,x_k>0$,
$$
\sum_{i=1}^k
\frac{x_i}{x_i+x_{i+1}+\cdots+x_k+1}\le\sum_{i=k}^1\int_{x_{i+1}+\cdots+x_k}^{x_{i}+\cdots+x_k}
\frac{dt}{t+1}=\int_0^{x_1+\cdots+x_k}\frac{dt}{t+1}=\log(x_1+\cdots+x_k+1)
$$
Thus
\begin{eqnarray}\label{eq:pass to log}
\sum_{\substack{i\in\{1,\ldots,m_v-1\}\\s_i(u)\neq s_i(v)}
}\frac{d_i(v)}{\sum_{j=i}^{m_v} d_j(v)}&=&
\sum_{\substack{i\in\{1,\ldots,m_v-1\}\\s_i(u)\neq s_i(v)}
}\frac{d_i(v)/d_{m_v}(v)}{\left(\sum_{j=i}^{m_v-1}
d_j(v)/d_{m_v}(v)\right)+1}\nonumber\\
&\le&
\log\left(1+\frac{1}{d_{m_v}(v)}\sum_{\substack{i\in\{1,\ldots,m_v-1\}\\s_i(u)\neq
s_i(v)} }d_i(v)\right).
\end{eqnarray}
Let $i$ be the smallest integer in $\{1,\ldots,m_v-1\}$ such that
$s_i(u)\neq s_i(v)$. Then by the definition of $s_i(\cdot)$,
$$
d_{m_v}(v)>\frac{\delta}{2}\sum_{h=i}^{m_v}d_h(v).
$$
It follows that
$$
\log\left(1+\frac{1}{d_{m_v}(v)}\sum_{\substack{i\in\{1,\ldots,m_v-1\}\\s_i(u)\neq
s_i(v)} }d_i(v)\right)\le
\log\left(1+\frac{1}{d_{m_v}(v)}\sum_{h=i}^{m_v}d_h(v)\right)\le
\log\left(1+\frac{2}{\delta}\right).
$$
Plugging this bound into~\eqref{eq:pass to log}, and
using~\eqref{eq:first sum} and~\eqref{eq:first bound}, we get that
$$
\|f(u)-f(v)\|_p\le
\left[4\log\left(1+\frac{2}{\delta}\right)+1\right]^{1/p}\cdot
d_T(u,v)\le \left[5\log(3/\delta)\right]^{1/p}\cdot d_T(u,v).
$$
\end{proof}

Our final claim bounds $\|f^{-1}\|_{\Lip}$.

\begin{claim}\label{claim:f inverse}
The embedding $f$ is invertible, and $\|f^{-1}\|_{\Lip}\le 48$.
\end{claim}
\begin{proof}
Fix $u,v\in V$, $u\neq v$, and let $j$ be the integer satisfying
$k_1(u)=k_1(v),\ldots,k_j(u)=k_j(v)$ and $k_{j+1}(u)\neq
k_{j+1}(v)$. It follows that
$d_1(u)=d_1(v),\ldots,d_{j-1}(u)=d_{j-1}(v)$, and we may assume
without loss of generality that $d_j(u)\ge d_j(v)$. With this
notation (see Figure~\ref{figure:tree-matousek} below),
\begin{eqnarray}\label{eq:formula for distance}
d_T(u,v)=
d_j(u)-d_j(v)+\sum_{i=j+1}^{m_u}d_i(u)+\sum_{i=j+1}^{m_v} d_i(v).
\end{eqnarray}
\begin{figure}\label{figure:tree-matousek}
 \centerline{\hbox{
        \psfig{figure=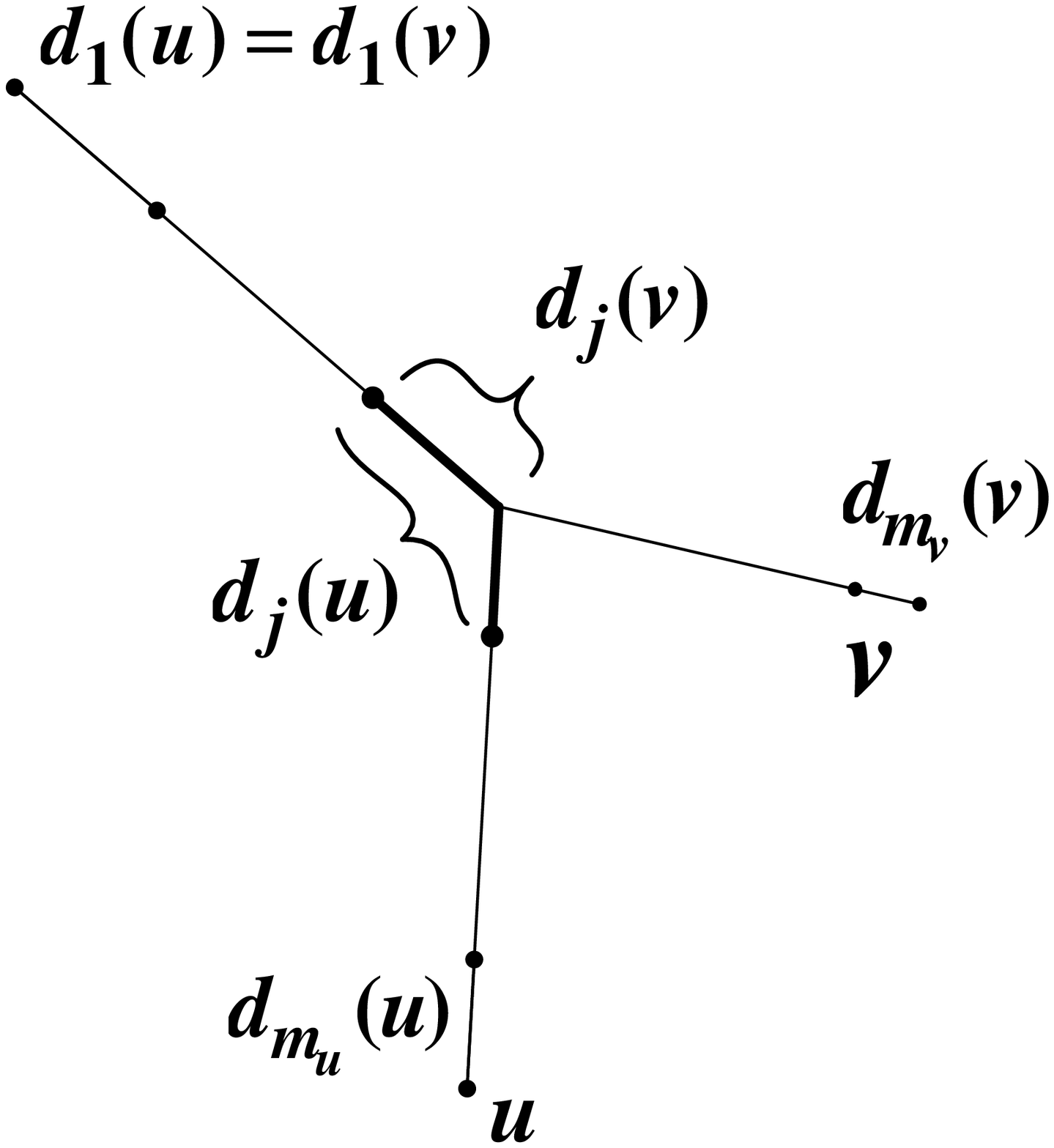,height=2.5in}
        }}
        \caption{A schematic description of the location of $u$ and $v$ in the tree $T$. The bold segment corresponds to the color class $k_j(u)=k_j(v)$.}\label{figure:net}
\end{figure}

On the other hand,
\begin{eqnarray}\label{eq:formula for f}
\|f(u)-f(v)\|_p^p&\ge&\nonumber
\left|[d_j(u)]^{1/p}[s_j(u)]^{(p-1)/p}-[d_j(v)]^{1/p}[s_j(v)]^{(p-1)/p}\right|^p+\\&\phantom{\le}&
\sum_{i=j+1}^{m_u}d_i(u)[s_i(u)]^{p-1}+\sum_{i=j+1}^{m_v}d_i(v)[s_i(v)]^{p-1}.
\end{eqnarray}
Using Claim~\ref{claim:use delta strong} we see that
\begin{eqnarray}\label{eq:bound first sum}
\sum_{i=j+1}^{m_u}d_i(u)[s_i(u)]^{p-1}&\ge&\nonumber\frac{1}{4^{p-1}}\sum_{i=j+1}^{m_u}
d_i(u)\left(\sum_{h=i}^{m_u} d_h(u)\right)^{p-1}\\\nonumber&\ge&
\frac{1}{4^{p-1}}\sum_{i=j+1}^{m_u}\int_{d_{i+1}(u)+\cdots
+d_{m_u}(u)}^{d_{i}(u)+\cdots +d_{m_u}(u)} t^{p-1}dt\\\nonumber
&=&\frac{1}{4^{p-1}} \int_0^{d_{j+1}(u)+\cdots+
d_{m_u}(u)}t^{p-1}dt\\&=&\frac{1}{p4^{p-1}}\cdot\left(\sum_{i=j+1}^{m_u}
d_i(u)\right)^{p}.
\end{eqnarray}
Similarly,
\begin{eqnarray}\label{eq:bound second sum}
\sum_{i=j+1}^{m_v}d_i(v)[s_i(v)]^{p-1}\ge
\frac{1}{p4^{p-1}}\cdot\left(\sum_{i=j+1}^{m_v} d_i(v)\right)^{p}.
\end{eqnarray}

We now consider two cases:

\medskip

\noindent{\bf Case 1.}
$\frac{d_j(u)-d_j(v)}{2}\le\sum_{i=j+1}^{m_v} d_i(v)$. In this
case, using~\eqref{eq:formula for distance} we see that
\begin{eqnarray*}
[d_T(u,v)]^p&\le&
3^p\left(\sum_{i=j+1}^{m_u}d_i(u)+\sum_{i=j+1}^{m_v}
d_i(v)\right)^p\\&\le& 3^p\cdot 2^{p-1}
\left(\sum_{i=j+1}^{m_u}d_i(u)\right)^p+ 3^p\cdot
2^{p-1}\left(\sum_{i=j+1}^{m_v}d_i(v)\right)^p\\&\le&
p4^{p-1}\cdot 3^p\cdot 2^{p-1}\|f(u)-f(v)\|_p^p,
\end{eqnarray*}
where in the last inequality we plugged the bounds~\eqref{eq:bound
first sum} and~\eqref{eq:bound second sum} into~\eqref{eq:formula
for f}. Thus we get that
$$
\|f(u)-f(v)\|_p\ge \frac{1}{48}\cdot d_T(u,v),
$$
as required.

\medskip
\noindent{\bf Case 2.} $\frac{d_j(u)-d_j(v)}{2}>\sum_{i=j+1}^{m_v}
d_i(v)$. In this case we observe that
$$
s_j(u)=\sum_{i=j}^{m_u}\max\left\{0,d_i(u)-\frac{\delta}{2}\sum_{h=j}^id_h(u)\right\}\ge
\left(1-\frac{\delta}{2}\right)d_j(u),
$$
and similarly
$$
 s_j(v)\le
\left(1-\frac{\delta}{2}\right)d_j(v)+\sum_{i=j+1}^{m_v}d_i(v)\le
\left(1-\frac{\delta}{2}\right)d_j(v)+\frac{d_j(u)-d_j(v)}{2}.
$$
Thus
\begin{eqnarray*}
&&\!\!\!\!\!\!\!\!\!\!\!\!\!\!\!\!\![d_j(u)]^{1/p}[s_j(u)]^{(p-1)/p}-[d_j(v)]^{1/p}[s_j(v)]^{(p-1)/p}\ge
\left(1-\frac{\delta}{2}\right)^{(p-1)/p}d_j(u)-\\&\phantom{\le}&\left(1-\frac{\delta}{2}\right)^{(p-1)/p}d_j(v)\cdot
\left(1+\frac{d_j(u)-d_j(v)}{(2-\delta)d_j(v)}\right)^{(p-1)/p}\\
&\ge&\left(1-\frac{\delta}{2}\right)^{(p-1)/p}d_j(u)-\left(1-\frac{\delta}{2}\right)^{(p-1)/p}d_j(v)\cdot
\left(1+\frac{d_j(u)-d_j(v)}{(2-\delta)d_j(v)}\right)\\&=&
\left(1-\frac{\delta}{2}\right)^{(p-1)/p}\cdot\frac{1-\delta}{2-\delta}\cdot
[d_j(u)-d_j(v)]\\
&\ge&\frac{d_j(u)-d_j(v)}{4}.
\end{eqnarray*}
\remove{ Observe that
\begin{eqnarray*}
&&\!\!\!\!\!\!\!\!\!\!\!\!\!\!\!\!\!\left|[d_j(u)]^{1/p}[s_j(u)]^{(p-1)/p}-[d_j(v)]^{1/p}[s_j(v)]^{(p-1)/p}\right|\ge
\left|[d_j(u)]^{1/p}-[d_j(v)]^{1/p}\right|\cdot
[s_j(u)]^{(p-1)/p}-\\&\phantom{\le}&
\left|[s_j(u)]^{(p-1)/p}-[s_j(u)]^{(p-1)/p}\right|\cdot
[d_j(u)]^{1/p}\\&\ge&
\frac{d_j(u)-d_j(v)}{p[d_j(v)]^{(p-1)/p}}\cdot
\left(1-\frac{\delta}{2}\right)^{(p-1)/p}[d_j(u)]^{(p-1)/p}-\frac{|s_j(u)-s_j(v)|}{[s_j(u)]^{1/p}}\cdot
[d_j(u)]^{1/p}\\&\ge& \frac{d_j(u)-d_j(v)}{2p}-2|s_j(u)-s_j(v)|,
\end{eqnarray*}
} where we used the fact that $\delta\le \frac12$.
Using~\eqref{eq:formula for f} and the bounds~\eqref{eq:bound
first sum} and~\eqref{eq:bound second sum}, it follows that
\begin{eqnarray*}
\|f(u)-f(v)\|_p^p&\ge&
\frac{1}{4^p}[d_j(u)-d_j(v)]^p+\frac{1}{p4^{p-1}}\cdot\left(\sum_{i=j+1}^{m_u}
d_i(u)\right)^{p}+\frac{1}{p4^{p-1}}\cdot\left(\sum_{i=j+1}^{m_v}
d_i(v)\right)^{p}\\&\ge&\frac{1}{p4^p\cdot
3^{p-1}}\left(d_j(u)-d_j(v)+\sum_{i=j+1}^{m_u}d_i(u)+\sum_{i=j+1}^{m_v}
d_i(v)\right)^p\\
&\ge& \frac{1}{24^{p}}\cdot[d_T(u,v)]^p.
\end{eqnarray*}

\end{proof}
Claim~\ref{claim:f inverse}, together with Claim~\ref{claim:flip},
concludes the proof of Theorem~\ref{thm:matousek}.
\end{proof}

\subsection{Relating coloring bounds to the containment of large binary
trees}\label{sec:binary}

The following theorem, in conjunction with
Theorem~\ref{thm:matousek} and Lemma~\ref{lem:relation}, implies
Theorem~\ref{thm:upper-binary}.  If one is concerned with simply
giving {\em some} upper bound on $c_p(T)$ in terms of $\mathscr{B}_T(c)$,
then it suffices to combine the following theorem with Lemma \ref{claim:color}.

\begin{theorem}\label{thm:upperB} For every weighted tree $T=(V,E)$ and
every $c>1$,
$$
\mathscr{B}_T(c)\ge \frac{c-1}{250c}\cdot\frac{1}{\e^*(T)}.
$$
\end{theorem}

\begin{proof} We start by introducing some notation. For a vertex $v\in V$ we
denote by $\mathscr C(v)$ the set of all children of $v$ in $T$.
Given $u\in \mathscr{C}(v)$ we denote by $T_u=(V_u,E_u)$ the
subtree rooted at $u$. We also let $F_u$ denote the tree
$F_u=(V_u\cup\{v\},E_u\cup\{(v,u)\})$, i.e. $F_u$ is $T_u$ plus
the ``incoming" edge $(v,u)$.

Recall that $B_k=(V_k,E_k)$ is the complete binary tree of height
$k$. Let $r_k$ be the root of $B_k$, and define an auxiliary tree
$M_k$ by $M_k=\left(V_k\cup \{m_k\},E_k\cup\{(m_k,r_k)\}\right)$
(i.e. $M_k$ is $B_k$ with an extra incoming edge). Given a connected
subtree $H$ of $T$ rooted at $r_H$, we shall say that $H$ {\em
admits a copy of $M_k$ at scale $j$} if there exits a one-to-one
mapping $f:M_k\to H$ such that
\begin{enumerate}
\item $f(m_k)=r_H$

\item $\|f\|_{\Lip}\le \frac{9c}{c-1}\cdot 4^{j}$ and $\|f^{-1}\|_{\Lip}\le \frac{c-1}{9\cdot4^{j}}$ (thus in
particular $\dist(f)\le c$).

\end{enumerate}
We define
$$
\mu_j(H)=\max\left\{k\in \mathbb N:\ H\ \mathrm{admits\ a\ copy\
of}\ M_k\ \mathrm{at\ scale\ }j  \right\},
$$
or $\mu_j(H)=-1$ if no such $k$ exists.

We shall now define a function $g:V\to \Z\cup\{\infty\}$ and a
coloring $\chi:E\to \Z$. These mappings will be constructed by
induction as follows. We start by setting $g(r)=\infty$. Assume
inductively that the construction is done so that whenever $v\in V$
is such that $g(v)$ is defined, if $u$ is a vertex on the path
$P(v)$ then $g(u)$ has already been defined, and for every edge
$e\in E$ incident with $v$, $\chi(e)$ has been defined.

Let $v\in V$ be a vertex closest to the root $r$ for which $g(v)$
hasn't yet been defined. Then, by our assumption, for every $e\in
P(v)$, $\chi(e)$ has been defined, and for every vertex $u$ other
than $v$ lying on the path $P(v)$, $g(u)$ has been defined. Let
$\beta_\chi(v)\subseteq V$ denote the set of {\em breakpoints} of
$\chi$ in $P(v)$, i.e. the set of vertices on $P(v)$ for which the
incoming and outgoing edges have distinct colors (for convenience,
in what follows we shall also consider the root $r$ as a breakpoint
of $\chi$). We define
$$
g(v)=\max\left\{j\in \mathbb Z:\ \forall\ u\in \beta_\chi(v),\
d_T(u,v)\ge 4^{\min \{g(u),j\}}\right\}.
$$
Having defined $g(v)$ we choose one of its children $w\in
\mathscr{C}(v)$ for which
$$
\mu_{g(v)} (F_w)=\max_{z\in \mathscr{C}(v)} \mu_{g(v)}(F_z).
$$
Letting $u$ be the father of $v$ on the path $P(v)$, we set
$\chi(v,w)=\chi(u,v)$, and we assign arbitrary new (i.e. which
haven't been used before) distinct colors to each of the edges
$\{(v,z)\}_{z\in \mathscr C(v)\setminus \{w\}}$. In other words,
given the ``scale" $j=g(v)$ we order the children of $v$ according
to the size of the copy of $M_k$ which they admit beneath them at
scale $j$. We then continue coloring with the color $\chi(u,v)$
the path $P(v)$ along the edge joining $v$ and its child which
admits the largest $M_k$ at scale $j$, and color the remaining
edges incident with $v$ by arbitrary distinct new colors.

This definition clearly results in a monotone coloring $\chi$. To
motivate this somewhat complicated construction, we shall now
prove some of the crucial properties of $\chi$ and $g$ which will
be used later.

\begin{claim}\label{claim:g}
Let $P$ be any monotone path in $T$, and let
$(b_1,b_2,\ldots,b_m)$ be the sequence of breakpoints along $P$
ordered down the tree (i.e. in increasing distance from the root).
Assume that $j\in \Z$ satisfies for every $i\in \{2,\ldots,m\}$,
$d_T(b_i,b_{i-1})\le 4^{j}$, and assume also that $d_T(b_1,b_m)\ge
\frac{30c}{c-1}\cdot 4^{j}$. Then there exists a subsequence of
the indices $1\le i_1<i_2<\cdots<i_k\le m$ such that
\begin{enumerate}
\item $k\ge \frac{c-1}{20c\cdot4^{j}}\cdot d_T(b_1,b_m)$.
\item For every $s\in
\{1,\ldots,k\}$ we have $g(b_{i_s})=j$.
\item For every $s\in
\{1,\ldots,k-1\}$ we have $\frac{9}{c-1}\cdot 4^{j}\le
d_T(b_{i_{s+1}},b_{i_s})\le \frac{9c}{c-1}\cdot 4^{j}$.
\end{enumerate}
\end{claim}

\begin{proof} We shall show that if $i\in \{1,\ldots,m\}$ is such
that $d_T(b_i,b_m)>\frac{4^{j+1}}{3}$ then there exists an index
$t\in \{1,\ldots,m\}$ with $g(b_t)=j$ and $d_T(b_t,b_i)\le
4^{j+1}$. Assuming this fact for the moment, we conclude the proof
as follows. Let $i_1$ be the smallest integer in $\{2,\ldots,m\}$
such that $g(b_{i_1})=j$. Then $d_T(b_{i_1},b_1)\le  4^{j+1}$.
Assuming we defined $i_1<i_2<\cdots<i_s$, if $d_T(b_{i_s},b_m)\le
\frac{9c}{c-1}\cdot 4^{j}$ we stop the construction, and otherwise
we let $t$ be the smallest integer bigger than $i_s$ such that
$d_T(b_{t},b_{i_s})\ge \frac{4c+5}{c-1}\cdot 4^{j}$. Since
$d_T(b_{t-1},b_{i_s})<\frac{4c+5}{c-1}\cdot 4^{j}$, we know that
$d_T(b_t,b_{i_s})<\frac{4c+5}{c-1}\cdot 4^{j}+4^j$. Thus
$d_T(b_t,b_m)> d_T(b_{i_s},b_m)- \frac{4c+5}{c-1}\cdot 4^{j}-4^j>
\frac{4^{j+1}}{3}$ (because we are assuming that
$d_T(b_{i_s},b_m)< \frac{9c}{c-1}\cdot 4^{j}$). So, there exists
$i_{s+1}\in \{1,\ldots,m\}$ such that $g(b_{i_{s+1}})=j$ and
$d_T(b_{i_{s+1}},b_t)\le  4^{j+1}$. Since by construction
$d_T(b_t,b_{i_s})\ge \frac{4c+5}{c-1}\cdot 4^{j}>4^{j+1}$ we
deduce that $i_{s+1}>i_s$ and
$$
\frac{9}{c-1}\cdot 4^{j}\le d_T(b_{t},b_{i_s})-
d_T(b_{i_{s+1}},b_{t})\le d_T(b_{i_{s+1}},b_{i_s})\le
d_T(b_{i_{s+1}},b_{t})+d_T(b_{t},b_{i_s}) \le \frac{9c}{c-1}\cdot
4^{j}.
$$
 This construction terminates after $k$ steps, in which
case we have that
$$
d_T(b_1,b_m)=d_T(b_1,b_{i_1})+\sum_{s=1}^{k-1}d_T(b_{i_s},b_{i_{s+1}})+d_T(b_{i_k},b_m)\le
 4^{j+1}+ (k-1)\cdot\frac{9c}{c-1}\cdot
4^{j}+\frac{9c}{c-1}\cdot 4^{j}.
$$
Since $d_T(b_1,b_m)\ge \frac{30c}{c-1}\cdot 4^{j}$, this implies
the required result.

\medskip

It remains to show that if $i\in \{1,\ldots,m\}$ is such that
$d_T(b_i,b_m)>\frac{4^{j+1}}{3}$ then there exists  $t\in
\{1,\ldots,m\}$ with $g(b_t)=j$ and $d_T(b_t,b_i)\le 4^{j+1}$. We
first claim that for every $i\in \{1,\ldots,m\}$ there is a
breakpoint $w\in \beta_\chi(b_i)$ with $g(w)\ge j$ and
$d_T(w,b_i)< \frac{4^{j+1}}{3}$. Indeed, if $g(b_i)\ge j$ then
there is nothing to prove, so assume that $g(b_i)<j$. By the
definition of $g$ there exists a breakpoint $w_1\in
\beta_\chi(b_i)$ such that
$$
4^{\min\{g(w_1),g(b_i)\}}\le
d_T(w_1,b_i)<4^{\min\{g(w_1),g(b_i)+1\}}.
$$
Thus necessarily $g(w_1)\ge g(b_i)+1$ and $d_T(w_1,b_i)<
4^{g(b_i)+1}< 4^{j}$. If $g(b_i)+1\ge j$ then we are done by
taking $w=w_1$. Otherwise, continuing in this manner we find a
breakpoint
 $w_2\in \beta_\chi(w_1)\subseteq \beta_\chi(b_i)$ with $g(w_2)\ge
g(w_1)+1\ge g(b_i)+2$ and $d_T(w_2,w_1)< 4^{g(w_1)+1}$. This
procedure terminates when we find a sequence
$b_i=w_0,w_1,w_2,\ldots,w_t$ with $g(w_t)\ge j$, $g(w_{t-1})<j$,
and for every $0\le s\le t-1$, $g(w_{s+1})\ge g(w_s)+1$ and
$d_T(w_{s+1},w_s)< 4^{g(w_s)+1}$. Thus
$$
d_T(b_i,w_t)=\sum_{s=0}^{t-1}d_T(w_{s+1},w_s)< \sum_{s=0}^{t-1}
4^{g(w_s)+1}< \sum_{s=-\infty}^j 4^s=\frac{4^{j+1}}{3}.
$$

Now, assume that $d_T(b_i,b_m)> \frac{4^{j+1}}{3}$. Let $s$ be the
largest integer in $\{i+1,\ldots,m\}$ such that $d_T(b_s,b_i)\le
\frac{4^{j+1}}{3}$ (such an integer $s$ exists since
$d_T(b_i,b_{i+1})\le 4^{j}$). Then $\frac{4^{j+1}}{3}<
d_T(b_{s+1},b_i)\le \frac{4^{j+1}}{3}+4^j$. By the previous
argument there is a break point $w\in \beta_\chi(b_{s+1})$ with
$g(w)\ge j$ and $d_T(w,b_{s+1})< \frac{4^{j+1}}{3}$. This implies
that $w=b_t$ for some $t\in \{i+1,\ldots,s+1\}$, and
$d_T(b_i,b_t)\le \frac{4^{j+1}}{3}+4^{j}$.

We proved that as long as $b_i$ satisfies $d_T(b_i,b_m)>
\frac{4^{j+1}}{3}$, there are $1\le t\le i\le s\le m$ such that
$g(b_s)\ge j$, $g(b_t)\ge j$, and $d_T(b_t,b_i)\le
\frac{4^{j+1}}{3}$, $d_T(b_s,b_i)\le \frac{4^{j+1}}{3}+4^j$. Thus,
by the definition of $g$,
$$
4^{\min\{g(b_s),g(b_t)\}}\le d_T(b_s,b_t)\le \frac{2\cdot
4^{j+1}}{3}+4^j<4^{j+1}.
$$
It follows that either $g(b_s)=j$ or $g(b_t)=j$, as required.
\end{proof}

To conclude the proof of Theorem~\ref{thm:upperB} we may assume that
$\e^*(T)<\frac{c-1}{240c}$, since otherwise the assertion of
Theorem~\ref{thm:upperB} is trivial. Fix
$\frac{c-1}{240c}>\e>\e^*(T)$. By the definition of $\e^*(T)$, the
coloring $\chi$ constructed above is not $\e$-good. Thus, there
exist two vertices $u,v\in V$ such that the path $P(u,v)$ does not
contain a monochromatic segment of length at least $\e d_T(u,v)$. We
may assume without loss of generality that $u$ is an ancestor of
$v$, and let $(b_1,b_2,\ldots,b_m)$ be the sequence of breakpoints
along this path, enumerated down the tree (i.e. from $u$ to $v$, not
necessarily including $u$ or $v$). Denoting $D=d_T(u,v)$ we have
that $d_T(u,b_1),d_T(v,b_m),d_T(b_i,b_{i+1})\le \e D$ for all $i\in
\{1,\ldots,m-1\}$. Fix $j\in \Z$ such that $\e D\le 4^j\le 4\e D$.
This choice implies that $d_T(b_i,b_{i+1})\le 4^j$ and
$d_T(b_1,b_m)\ge (1-2\e)D\ge \frac{1-2\e}{4\e}\cdot 4^j\ge
\frac{30c}{c-1}\cdot 4^{j}$. By Claim~\ref{claim:g} there is an
integer $k\ge \frac{(c-1)(1-2\e)D}{20c\cdot 4^j}\ge
\frac{c-1}{250c}\cdot\frac{1}{\e}+2$ (using the upper bound on $\e$)
and a sequence of breakpoints $s_1,\ldots,s_k$ on the path $P(u,v)$
(ordered down the tree) such that $g(s_1)=\cdots=g(s_k)=j$ and for
$i\in \{1,\ldots,k-1\}$, $\frac{9}{c-1}\cdot 4^{j}\le
d_T(s_i,s_{i+1})\le\frac{9c}{c-1}\cdot 4^{j}$.

The proof of Theorem~\ref{thm:upperB} will be complete once we show
that $\mathscr{B}_T(c)\ge k-2$. For $i\in \{1,\ldots, k\}$ let $t_i$
be the child of $s_{i}$ along the path $P(u,v)$. We will prove by
reverse induction on $i\in \{1,\ldots,k-1\}$ that
$\mu_j(F_{t_{i}})\ge k-i-1$, implying the required result. The base
case is true, i.e.  $\mu_j(F_{t_{k-1}})\ge 0$, since the pair
$(s_{k-1},s_k)$ constitutes a copy of $M_0$ at scale $j$.

Assuming that $\mu_j(F_{t_{i}})\ge k-i-1$ we shall prove that
$\mu_j(F_{t_{i-1}})\ge k-i$. Since $s_i$ was a breakpoint, the
construction of $\chi$ implies that there must be a child $t_i'$ of
$s_i$, {\em other than $t_i$}, for which
$\mu_j(F_{t_i'})>\mu_j(F_{t_{i}})\ge k-i-1$. Thus, there exist one
to one mappings $f,f':M_{k-i-1}\to T$ such that
$f(m_{k-i-1})=f'(m_{k-i-1})=s_i$, $f(M_{k-i-1})\subseteq F_{t_{i}}$,
$f'(M_{k-i-1})\subseteq F_{t'_{i}}$, $\|f\|_{\Lip},\|f'\|_{\Lip}\le
\frac{9c}{c-1}\cdot 4^{j}$, and
$\|f^{-1}\|_{\Lip},\|(f')^{-1}\|_{\Lip}\le \frac{c-1}{9\cdot4^{j}}$.
Thinking of $M_{k-i}$ as two disjoint copies of $M_{k-i-1}$, joined
at the root $m_{k-i}$, we may ``glue" $f$ and $f'$ to an embedding
$\overline{f}$ of $M_{k-i}$ by setting
$\overline{f}(m_{k-i})=s_{i-1}$. Since $\frac{9}{c-1}\cdot 4^{j}\le
d_T(s_i,s_{i-1})\le\frac{9c}{c-1}\cdot 4^{j}$, this results in an
embedding at scale $j$ of $M_{k-i}$  into $F_{t_{i-1}}$, as required
(see Figure~\ref{figure:gluing}).
\begin{figure}
 \centerline{\hbox{
        \psfig{figure=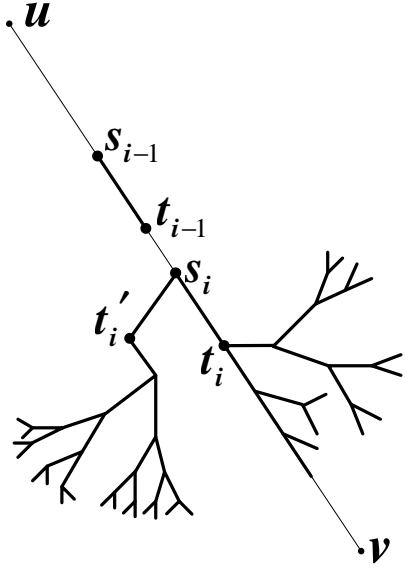,height=3in}
        }}
        \caption{A schematic description of the gluing procedure in the inductive step. Because $s_i$ was a breakpoint it must have two copies of $M_{k-i-1}$ at scale $j$ below it.}\label{figure:gluing}
\end{figure}
\end{proof}

\subsection{Embedding into finite-dimensional spaces}
\label{sec:doubling}

 We recall that the {\em doubling constant $\lambda(X)$}
 of a metric space
$X$ is the infimal value of $\lambda$ for which every ball
in $X$ can be covered by $\lambda$ balls of half the radius.
If $S \subseteq X$ is a $\delta$-separated set in $X$, then
a standard observation is that
$|S| \leq \lambda(X)^{O(\diam(S)/\delta)}$.
This section is devoted to a simpler proof of the following theorem
of Gupta, Krauthgamer, and Lee originally proved in \cite{GKL03}.
(We stress that the only results we need for this section
are Lemma \ref{claim:color} and Theorem \ref{thm:upperB}.)

\begin{theorem}[\cite{GKL03}]\label{thm:gkl}
A tree metric $T$ embeds into a finite-dimensional Euclidean
space if and only if $\lambda(T) < \infty$.  In other words,
every doubling tree $T$ admits a $D$-embedding into $\mathbb R^k$ with
$D, k$ depending only on $\lambda(T)$.
\end{theorem}

Let $T = (V,E)$ be a weighted, rooted tree. Note that the ``only if"
part of Theorem~\ref{thm:gkl} is trivial. In order to prove the
remaining implication we need a coloring notion weaker than
$\varepsilon$-good. Let $\chi : E \to \mathbb Z$ be a coloring of
the edges of $T$ which is {\em not necessarily monotone.}  We will
say that $\chi$ is {\em $\varepsilon$-reasonable} if the following
holds for every $u,v \in V$.  Let $w = \lca(u,v)$, and recall that
$P(w,u), P(w,v)$ denote the paths from $w$ to $u$ and $v$,
respectively. Then there should exist a color $c \in \mathbb Z$ for
which
\begin{equation}\label{eq:reasonable}
\left | \sum_{e \in P(w,u) : \chi(e) = c} \ell(e) - \sum_{e \in P(w,v) : \chi(e) = c} \ell(e) \right | \geq \varepsilon d_T(u,v).
\end{equation}

Since a reasonable coloring is not necessarily monotone, it is possible to
construct such colorings where $\chi^{-1}(E)$ is finite even though
$T$ might be infinite.  The number of colors used, i.e. $|\chi^{-1}(E)|$, controls
the dimension of the  embedding from
Lemma \ref{claim:color}.

\begin{lemma}\label{claim:color2}
Let $T = (V,E)$ be a weighted tree, and suppose that $T$ admits
an $\varepsilon$-reasonable coloring for some $\varepsilon > 0$.
Then $T$ embeds in $\mathbb R^k$ (equipped, e.g. with the $L_2$ norm)
with distortion $O(1/\varepsilon)$, and $k = |\chi^{-1}(E)|$.
\end{lemma}

\begin{proof}
Let $\chi :E \to \mathbb Z$ be an $\varepsilon$-reasonable coloring of $T$.
We use the embedding $f : V \to \ell_2$ of Lemma \ref{claim:color}.  In particular,
it is easy to check that the definition of the embedding
does not require $\chi$ to be monotone.  Observe that $\mathrm{Im}(f)$ lies
naturally in $\mathrm{span} \{ e_k : k \in \chi^{-1}(E) \}$, and thus
we may assume that $f : V \to \mathbb R^k$ with $k = |\chi^{-1}(E)|$.

From the proof of Lemma \ref{claim:color}, we conclude that
$\|f\|_\Lip \leq 1$, and thus we need only consider
$\|f^{-1}\|_\Lip$.  But it is easy to see that condition
\eqref{eq:reasonable} suffices to obtain a similar lower bound in
equation \eqref{eq:goodcolor} of Lemma \ref{claim:color}.
\end{proof}

We note that the dependence of $k$ on $|\chi^{-1}(E)|$ in the above
lemma can be improved to $k = O(\log |\chi^{-1}(E)|)$ using
a ``nearly-orthogonal'' set of vectors instead of the orthonormal set $\{e_k\}_{k \in \mathbb Z}$.
We refer to \cite{GKL03} for details.

Now, clearly $\mathscr{B}_T(2) \leq O(\log \lambda(T))$ since $\lambda(B_m) = 2^{\Theta(m)}$, hence
$\varepsilon^*(T) \geq 1/O(\log \lambda(T))$ using Theorem \ref{thm:upperB}.
In light of Lemma \ref{claim:color2} and the preceding remark,
the following result completes the proof of Theorem \ref{thm:gkl}.
(Note that we can assume $T$ finite by compactness---a tree embeds
into a finite-dimensional Euclidean space if and only if every
finite subset embeds with uniformly bounded distortion).

\begin{theorem}
Let $T = (V,E)$ be a finite, weighted tree.  If $T$ admits an $\varepsilon$-good
coloring, then it also admits an $O(\varepsilon)$-reasonable coloring
with only $\lambda(T)^{(1/\varepsilon)^{O(1/\varepsilon)}}$ colors.
\end{theorem}

\begin{proof}
We will say that a monotone coloring $\chi :E \to \mathbb Z$ is {\em regular} if
the following holds:  For every maximal monochromatic segment
$s = \{e_1, e_2, \ldots, e_k\} \subseteq E$ (with edges ordered down the tree),
and for every $1 \leq i \leq k$,
we have $\ell(e_{i+1}) \leq 2 \sum_{j=1}^i \ell(e_i)$.

\begin{lemma}
If a finite tree $T$ admits an $\varepsilon$-good coloring, then $T$ admits
an $O(\varepsilon)$-good regular coloring.
\end{lemma}

\begin{proof}
Let $T = (V,E)$ be a rooted tree, and let $\chi_0 : E \to \mathbb Z$ be an
$\varepsilon$-good coloring of $T$.
Suppose that some monochromatic segment $s = \{e_1, e_2, \ldots, e_k\} \subseteq E$
violates the regularity condition.  Let $i \in [k]$ be the smallest index for which
$\ell(e_{i+1}) > 2 \sum_{j=1}^i \ell(e_i)$.  We derive a new coloring $\chi_1 : E \to \mathbb Z$
by coloring the edges $e_1, \ldots, e_i$ with a new unused color $c \in \mathbb Z$,
i.e. $\chi_1(e) = c$ if $e = e_j$ for some $1 \leq j \leq i$ and $\chi_1(e) = \chi_0(e)$ otherwise.
Continue this process inductively until the resulting coloring $\chi' : E \to \mathbb Z$ is regular.
This process terminates because $T$ is finite.
It remains to show that $\chi'$ is $O(\varepsilon)$-good.

To this end, let $s = \{e_1, \ldots, e_k\} \subseteq E$ be a maximal monochromatic segment
according to $\chi_0$, and let $s_1, s_2, \ldots, s_m \subseteq s$ be the maximal
monochromatic segments of $s$ according to $\chi'$, ordered down the tree.
By construction, we have
$$
\ell(s_m) \geq 2 \ell(s_{m-1}) \geq \ell(s_{m-1}) + 2\ell(s_{m-2}) \geq \cdots \geq \ell(s_1) + \cdots + \ell(s_{m-1}),
$$
hence $\ell(s_m) \geq \frac12 \ell(s)$.  It follows that $\chi'$ is a regular $\varepsilon/2$-good coloring of $T$.
\end{proof}

Let $T$ be a rooted tree, and let $\chi : E \to \mathbb Z$ be an
$\varepsilon$-good coloring of $T$.  Using the preceding lemma,
we may assume that $\chi$ is regular.
Let $\mathcal C$ be the set of
color classes.  We will think of segments $s \in \mathcal C$
sometimes as a subset of edges and sometimes a subset of vertices (the
endpoints and internal vertices of the segments), depending on the
context.  In everything that follows, we will assume that for $s \neq s' \in \mathcal C$,
we have $\diam(s) \neq \diam(s')$.  This is without loss of generality
by applying arbitrarily small perturbations to $T$.  (Alternatively, we could
fix a total order on segments of equal diameter, but this would add
unnecessary notation to the proof.)

For every segment $s \in \mathcal C$, we define $p(s)$ as the vertex of
$s$ which is closest to the root.
For every $s_0 \in \mathcal C$ and $K > 0$ we
define a relative length function
\begin{equation}\label{eq:relength}
\length_{s_0}(s;K) = \max \left\{ \diam\left(P(p(s), x)\right):
x \in s \cap B_T\left(\vphantom{\bigoplus}p(s_0), K\cdot \diam(s_0)\right) \right\},
\end{equation}
where we take $s \in \mathcal C$, and
we set $\length_{s_0}(s;K) = 0$ in case the maximum is empty.
In words, this is how long the segment $s \in \mathcal C$ ``looks'' from
$p(s_0)$, where the ``view'' is restricted to a ball of radius $K \cdot \diam(s_0)$.
It is important to note that even when $s \nsubseteq B_T(p(s_0), K \cdot \diam(s_0))$,
one might have $0 < \length_{s_0}(s;K) \ll K \cdot \diam(s_0)$ since $T$ is
not necessarily an $\mathbb R$-tree.

Now we define carefully a {\em directed}
graph $G_{\mathcal C} = (\mathcal C, E_{\mathcal C})$.
The adjacency relationship on $G_{\mathcal C}$ will be the key
in producing an $O(\varepsilon$)-reasonable coloring.
We put
\begin{equation}\label{eq:adj}
(s,s') \in E_{\mathcal C} \iff \length_{s}(s';K) > \diam(s),
\end{equation}
for some constant $K \geq 6$ to be chosen later.
Observe, in particular, that $(s,s') \in E_{\mathcal C} \implies \diam(s') > \diam(s)$.
We will argue
that the undirected graph $\hat G_{\mathcal C}$ which results from
ignoring the edge directions in $G_{\mathcal C}$ has its chromatic number
bounded solely by a function of $\lambda(T)$.
We accomplish
this with the following sequence of lemmas.
(This step is
non-trivial since $\hat G_{\mathcal C}$ does not have bounded degree.)
\begin{lemma}\label{lem:color2}
For every $s \in \mathcal C$, the out-degree is bounded, i.e.
$$
|\{s' \in \mathcal C : (s,s') \in E_{\mathcal C}\}| \leq \lambda(T)^{O(K)}.
$$
\end{lemma}

\begin{proof}
Fix $s \in \mathcal C$.  For every $s' \in \mathcal C$ with $(s,s') \in E_{\mathcal C}$,
let $x_{s'} \in s'$ be the node achieving the maximum in \eqref{eq:relength}.
If the maximum does not exist then $\length_{s}(s';K) = 0$, hence $(s,s') \notin E_{\mathcal C}$.
By definition, $d_T(p(s),x_{s'}) \leq K \cdot \diam(s)$.
Furthermore, $d_T(p(s'),x_{s'}) = \length_s(s';K) > \diam(s)$.  It follows
that the set $X_s = \left\{ x_{s'} : (s,s') \in E_{\mathcal C} \right\}$ is $\diam(s)$-separated.
Since $X_s \subseteq B_T(p(s), K\cdot \diam(s))$, the doubling property implies
that $$|\{s' \in \mathcal C : (s,s') \in E_{\mathcal C}\}| = |X_s| \leq \lambda(T)^{O(K)}.$$
\end{proof}

\remove{
\begin{lemma}\label{lem:color1}
For every $s \in \mathcal C$,
$$
\left|\left\{ s'  \in \mathcal C: (s',s) \in E_{\mathcal C} \textrm{ and } \diam(s') \geq \diam(s) \right\}\right| \leq
\lambda(T)^{O(K/\eta)}.
$$
\end{lemma}

\begin{proof}
Fix $s \in \mathcal C$.  For every $s' \in \mathcal C$ such that $(s',s) \in E_{\mathcal C}$, let
$x_{s'} \in s'$ be the point of $s'$ furthest from the root.
Define $X_s = \{ x_{s'} : (s',s) \in E_{\mathcal C} \textrm{ and } \diam(s') \geq \diam(s)\}$.
Using \eqref{eq:adj}, we see that $$\diam(s) \geq \length_{s'}(s;K) \geq \eta \cdot \diam(s').$$
Furthermore, since $(s',s) \in E_{\mathcal C}$, we have $d_T(p(s'), p(s)) \leq K \cdot \diam(s') + \diam(s)$.
It follows that for every $x_{s'} \in X_s$, we have
\begin{eqnarray*}
d_T(p(s), x_{s'}) &\leq& d_T(p(s),p(s')) + d_T(p(s'),x_{s'}) \\
&\leq& K \cdot \diam(s') + \diam(s) + \diam(s') \leq (K+2) \diam(s')
\leq \frac{K+2}{\eta}\, \diam(s).
\end{eqnarray*}
We conclude that $X_s \subseteq B_T\left(p(s), ((K+2)/\eta) \diam(s)\right)$.
On the other hand, the set $X_s$ is $\delta$-separated
where $\delta = \min \{ \diam(s') : x_{s'} \in X_s \} \geq \diam(s)$.
The doubling property implies that $|X_s| \leq \lambda(T)^{O(K/\eta)}$.
\end{proof}
}

For any undirected graph $G = (V_G, E_G)$ and $v \in V_G$, we
define $N(v)$ to be the set of neighbors of $v$ in $G$,
we let $\deg(v) = |N(v)|$ and $\deg_S(v) = |N(v) \cap S|$ for $S \subseteq V_G$.
The next result is well-known.

\begin{lemma}\label{lem:chrom}
Let $G = (V_G,E_G)$ be any finite, undirected graph.  Let
$k \in \mathbb N$
and let $\pi : V_G \to \{1,2,\ldots, n\}$ be a permutation.
We denote $\pi_j = \{ v \in V_G : \pi(v) \leq j \}$.
If, for every $j = 1, 2, \ldots, n$, we have
$$
\deg_{\pi_{j-1}}(\pi^{-1}(j)) \leq k,
$$
then the chromatic number of $G$ is at most $k+1$.
\end{lemma}

\begin{proof}
The proof follows by inductively coloring the elements $\pi^{-1}(1), \pi^{-1}(2), \ldots, \pi^{-1}(n)$ in order.
If we have a palette of $k+1$ colors, then
since $\deg_{\pi_{j-1}}(\pi^{-1}(j)) \leq k$, we can always choose a new color
for $\pi^{-1}(j)$ that doesn't conflict with any already colored vertex in $\pi_{j-1}$.
\end{proof}

\begin{corollary}
If $\hat G_{\mathcal C}$ is the undirected version of $G_{\mathcal C}$, then
the chromatic number of $\hat G_{\mathcal C}$ is bounded by $\lambda(T)^{O(K)}$.
\end{corollary}

\begin{proof}
Let $\pi : \mathcal C \to \{1,2,\ldots,|\mathcal C|\}$ be any permutation
for which $\diam(\pi(j)) \geq \diam(\pi(j+1))$ for $1 \leq j \leq |\mathcal C|-1$
(i.e. the diameters of the segments decrease monotonically).
Then combining Lemmas \ref{lem:color2} and the fact
that $(s,s') \in E_{\mathcal C} \implies \diam(s') > \diam(s)$ shows
that for $j = 1, 2, \ldots, n$,
$$
\deg_{\pi_{j-1}}(\pi^{-1}(j)) \leq \lambda(T)^{O(K)}.
$$
Applying Lemma \ref{lem:chrom} completes the proof.
\end{proof}

Now let $\chi_{\mathcal C} : \mathcal C \to [k]$ be
a proper coloring of $\hat G_{\mathcal C}$ using only $k = \lambda(T)^{O(K)}$ colors.
We are done as soon as we show that $\chi_{\mathcal C}$ is an $O(\varepsilon)$-reasonable
edge-coloring of $T$ (where we consider $\chi_{\mathcal C}$ as a coloring of $E$
in the obvious way) for some choice of $K \leq (1/\varepsilon)^{O(1/\varepsilon)}$.

\begin{lemma}\label{lem:bigsegs}
Suppose that for $s \neq s' \in \mathcal C$, we have
$$\diam(s \cap P(u,v)), \diam(s' \cap P(u,v)) \geq \frac{d_T(u,v)}{K/2-1},$$
where $u,v \in T$.
Then $\chi_{\mathcal C}(s) \neq \chi_{\mathcal C}(s')$.
\end{lemma}

\begin{proof}
Assume that $\diam(s') > \diam(s)$, and let
$x$ be the bottom-most point of $s' \cap P(u,v)$.
Then
\begin{equation}\label{eq:close1}
d_T(p(s), x) \leq \diam(s) + d_T(u,v) \leq \left(1+(K/2 -1)\right) \diam(s) \leq \frac{K}{2} \cdot \diam(s).
\end{equation}
In this case, $\length_{s}(s';K) \geq \diam(P(p(s'), x))$,
hence if $\diam(P(p(s'),x)) > \diam(s)$, we have $(s,s') \in E_{\mathcal C}$, which finishes
the proof of the lemma.

So we may assume that $\diam(P(p(s'),x)) \leq \diam(s)$.
We claim that in this case, $\length_{s}(s';K) > \diam(s)$
using the regularity of $\chi$.
Let $y \in s'$ be such that $d_T(x,y) \leq \frac{K}{2} \diam(s)$, and for which
$d_T(p(s'),y)$ is maximal.  If $d_T(p(s'),y) > \diam(s)$, then we are done
since by \eqref{eq:close1}, we have $d_T(p(s),y) \leq K \cdot \diam(s)$, implying
$\length_s(s';K) \geq d_T(p(s'),y) > \diam(s)$. Hence we may assume
that $d_T(p(s'),y) \leq \diam(s)$.  In this case, since $\diam(s') > \diam(s)$,
there exists an edge $(y,y')$ with $y' \in s'$ and $d_T(p(s'),y') > \frac{K}{2} \cdot \diam(s) > 3 \cdot \diam(s)$.
But this implies that $\ell(y,y') > 2 \cdot d_T(p(s'),y)$, which contradicts
the regularity of $\chi$.

It follows that $\length_s(s';K) > \diam(s)$, which again
implies $(s,s') \in E_{\mathcal C}$.
\end{proof}

Now fix $u,v \in V$ and $w = \lca(u,v)$, and suppose
that $d_T(w,u) \geq d_T(w,v)$.  Since $\chi$ is an $\varepsilon$-good coloring,
there exists a maximal monochromatic segment (with respect to $\chi$) $s \subseteq E$
for which $\diam(s \cap P(w,u)) \geq \varepsilon d_T(w,u) \geq (\varepsilon/2) d_T(u,v)$.
Now set $K = 4 \left(\frac{2}{\varepsilon}\right)^{1+2/\varepsilon}$.
Applying Lemma \ref{lem:bigsegs}, we see that for any $s' \subseteq E$
with $\chi_{\mathcal C}(s) = \chi_{\mathcal C}(s')$, we have
$\diam(s' \cap P(w,v)) \leq \left(\frac{\varepsilon}{2}\right)^{1+2/\varepsilon}$.
But now line \eqref{eq:geometric} of Lemma \ref{lem:relation} implies that
segments of this length can cover at most an $\varepsilon/2$-fraction of $P(w,v)$
(since $\chi$ is an $\varepsilon$-good coloring), which
is at most an $\varepsilon/4$-fraction of $P(u,v)$.
It follows that $\chi_{\mathcal C}$ is a $\delta$-reasonable coloring
for $\delta = \frac{\varepsilon}{2} - \frac{\varepsilon}{4} \geq \frac{\varepsilon}{4}$,
completing the proof.

\end{proof}

\section{Markov convexity and distortion lower bounds}

In this section we study Markov convexity, and show how it can be
used to prove several distortion lower bounds. In particular, we
will discuss the connection between Markov convexity and uniform
convexity in Banach spaces, and we will prove that
Theorem~\ref{thm:upper-binary} is optimal.

\subsection{Markov convexity in Banach spaces}\label{sec:banach}

We start by showing that Hilbert space is Markov $2$-convex. This
has essentially been proved by Bourgain in~\cite{Bou86-trees}. We
give the following proof here because the argument is extendable to
the case of $p \neq 2$. We refer also to~\cite{LS03} for another
variant of Bourgain's proof.

\begin{lemma}\label{lem:identity} For every $x_0,\ldots,x_{2^m}\in
L_2$,
\begin{equation}\label{eq:mtid}
\sum _{i=1}^{2^m} \|x_i - x_{i-1}\|_2^2 = {\|x_{2^m} -x_0\|_2^2 \over 2^m} + \sum
_{k=1}^m {1 \over 2^k} \sum _{j=1}^{2^{m-k}} \|x_{j2^k} -
2x_{(2j-1)2^{k-1}} + x_{(j-1)2^k}\|_2^2 \, .
\end{equation}
\end{lemma}
\begin{proof}
Let $\F_n$ be the $\sigma$-algebra of subsets of $[0,1]$ generated
by the intervals
$\left\{I^n_j:=\left[\frac{j-1}{2^n},\frac{j}{2^n}\right]\right\}_{j=1}^{2^n}$.
Define $\f:[0,1]\to L_2$ by $\f\equiv x_j-x_{j-1}$ on $I_j^m$. Set
$\f_j=\E (\f|\F_j)$, where the expectation is with respect to the
Lebesgue measure on $[0,1]$. In other words, for every $j\in
\{1,\ldots, 2^k\}$ and $t\in I_j^k$
$$
\f_k(t)=\frac{1}{2^{m-k}}\sum_{\ell=2^{m-k}(j-1)+1}^{j2^{m-k}}
(x_\ell-x_{\ell-1})=\frac{x_{j2^{m-k}}-x_{(j-1)2^{m-k}}}{2^{m-k}}.
$$
Since the sequence $\{\f_k-\f_{k-1}\}_{k=1}^m$ is a martingale
difference sequence, and $\f_0$ is constant, the functions
$\f_0,\f_1-\f_0,\f_2-\f_1,\ldots,\f_m-\f_{m-1}$ are orthogonal (in
the Hilbert space $L_2(L_2)$). Thus
$$
\E \| \f_m\|_2^2=\E\|\f_0\|_2^2+\sum_{k=1}^m
\E\|\f_k-\f_{k-1}\|_2^2.
$$
This is precisely the required identity.
\end{proof}

We can remove the dyadic bias from \eqref{eq:mtid} by averaging
over shifts.

\begin{corollary}\label{cor:nodyad}
For every $x_0, \ldots, x_{2^m} \in L_2$,
$$
\sum_{i=1}^{2^m} \|x_i - x_{i-1}\|_2^2 \geq \frac12 {\|x_{2^m} -x_0\|_2^2 \over 2^{2m}} + \frac12 \sum_{k=1}^m 2^{-2k}
\sum_{t=1}^{2^m} \|x_t - 2 x_{t-2^{k-1}} + x_{t-2^k}\|_2^2,
$$
where, by convention, $x_j = x_0$ for $j \leq 0$.
\end{corollary}

\begin{proof}
First, consider the sequence of length $3\cdot 2^m - 2$,
$$x_0, x_0, \ldots, x_0,\,\, x_0, x_1, \ldots, x_{2^m}, \,\, x_{2^m}, x_{2^m}, \ldots, x_{2^m},$$
which is the original sequence with $2^m-1$ copies of $x_0$ and $x_{2^m}$ appended
to the front and back, respectively.  Call this sequence $\{y_j\}_{j=1}^{3\cdot 2^m-2}$.
Now average the equality \eqref{eq:mtid} over all $2^{m+1}$ contiguous subsequences
of length $2^m+1$, i.e. $\{y_i, y_{i+1}, \ldots, y_{i+2^m}\}$ for $i = 1, \ldots, 2^{m+1}$.
By counting terms, this yields the desired result.
\end{proof}

\begin{theorem}\label{thm:hilbert has markov} Hilbert space is
Markov $2$-convex. In fact, $ \Pi_2(L_2)\le 4$.
\end{theorem}

\begin{proof} Let $\{X_t\}_{t=0}^\infty$ be a Markov chain on a
state space $\Omega$, and take $f:\Omega\to L_2$. By
Corollary \ref{cor:nodyad},
\begin{eqnarray*}
&& \sum_{t=1}^{2^m} \E\|f(X_t) - f(X_{t-1})\|_2^2 \ge
\frac12 \sum_{k=1}^m
2^{-2k} \sum _{t=1}^{2^{m}} \E\|f(X_{t}) -
2f(X_{t-2^{k-1}}) + f(X_{t-2^k})\|_2^2 \\
&&\qquad\qquad\qquad\qquad\qquad\qquad\qquad
+\, \frac12 \frac{\mathbb E \|f(X_0)-f(X_{2^m})\|_2^2}{2^{2m}},
\end{eqnarray*}
where by convention we set $X_t = X_0$ for $t \leq 0$.

Observe that for every two i.i.d. random vectors $Z,Z'\in L_2$,
and every constant $a\in L_2$, $\E\|Z-Z'\|_2^2\le 2\E\|Z-a\|_2^2$.
Thus, using the fact that conditioned on
$\mathcal X = (X_0,\ldots,X_{t-2^{k-1}})$ the random vectors $f(X_{t})$ and
$f(\widetilde X_{t}(t-2^{k-1}))$ are i.i.d., we see that
\begin{eqnarray*}
\E \|f(X_{t})-f(\widetilde
X_{t}(t-2^{k-1}))\|_2^2
&=&\E\left(\E\left(\|f(X_{t})-f(\widetilde
X_{t}(t-2^{k-1}))\|_2^2\Big|\mathcal X\right)\right)\\
&\le& 2\, \E\|f(X_{t}) - 2f(X_{t-2^{k-1}}) +
f(X_{t-2^k})\|_2^2.
\end{eqnarray*}
Likewise, $\mathbb E \|f(X_0)-f(\widetilde X_{2^m}(0))\|_2^2 \leq 2\, \mathbb E \|f(X_0)-f(X_{2^m})\|_2^2.$
It follows that
\begin{eqnarray*}
\sum_{t=1}^{2^m} \E\|f(X_t) - f(X_{t-1})\|_2^2 &\ge&
\frac1{4} \sum_{k=1}^{m+1}
2^{-2k} \sum_{t=1}^{2^{m}}
\E \|f(X_{t})-f(\widetilde
X_{t}(t-2^{k-1}))\|_2^2, \\
&=&
\frac1{16} \sum_{k=0}^m
2^{-2k} \sum_{t=1}^{2^{m}}
\E \|f(X_{t})-f(\widetilde
X_{t}(t-2^{k}))\|_2^2,
\end{eqnarray*}
completing the proof.
\end{proof}

\begin{remark}\label{rem:superreflexive}{\em
The above argument can be generalized to prove that $p$-convex
Banach spaces are Markov $p$-convex. Recall that a Banach space
$X$ is said to be $p$-convex with constant $K$ (see~\cite{BCL94})
if for every $x,y\in X$,
$$
2\|x\|^p+\frac{2}{K^p}\|y\|^p\le \|x+y\|^p+\|x-y\|^p.
$$
The least such constant $K$ is denoted $K_p(X)$.

 We claim that for
every Banach space $X$,
$$
\Pi_p(X)\le 2^{(p-1)/p}\left(2^{p-1}-1\right)^{1/p}\cdot K_p(X)\le
4K_p(X).
$$
Indeed, repeating the argument of Lemma~\ref{lem:identity}, we
replace the use of orthogonality by Pisier's
inequality~\cite{Pisier75} to get that (see the argument
in~\cite{Ball92} for the constant used below),
$$
(2^{p-1}-1)[K_p(X)]^p\E\|\f_m\|_X^p\ge \E\|\f_0\|_X^p+\sum_{k=1}^m
\E\|\f_k-\f_{k-1}\|_X^p.
$$
As in the proof of Theorem~\ref{thm:hilbert has markov} (and using
the notation there) this shows that
\begin{eqnarray*}
&& (2^{p-1}-1) [K_p(X)^p]\sum_{t=1}^{2^m} \E\|f(X_t) - f(X_{t-1})\|_X^p \ge \\
&& \qquad\qquad\qquad
\frac12 \sum_{k=1}^m
2^{-pk} \sum _{t=1}^{2^{m}} \E\|f(X_{t}) -
2f(X_{t-2^{k-1}}) + f(X_{t-2^k})\|_X^p
+\, \frac12 \frac{\mathbb E \|f(X_0)-f(X_{2^m})\|_X^p}{2^{pm}}.
\end{eqnarray*}

Since for every two i.i.d. random vectors $Z,Z'\in X$, and every
constant $a\in X$, we have that $\E\|Z-Z'\|_X^p\le
2^{p-1}\E\|Z-a\|_X^p$ (this fact follows from a straightforward
interpolation argument), we conclude exactly as in the proof of
Theorem~\ref{thm:hilbert has markov}.}
\end{remark}

We mention some partial converses to Remark \ref{rem:superreflexive}.

\begin{corollary}\label{coro:dvoretzky}
Let $X$ be an infinite dimensional Banach space. Then
$\Pi_p(X)<\infty$ implies that $X$ is superreflexive and has
cotype $q$ for every $q>p$.
\end{corollary}

\begin{proof} Let $q_X=\inf\{q:\ X\ \textrm{has\ cotype\ }q\}$. By
the Maurey-Pisier~\cite{MP75} theorem, $X$ contains copies of
$\ell_{q_X}^n$ with distortion uniformly bounded in $n$. By
Bourgain's embedding of trees into
$\ell_{q_X}$~\cite{Bou86-trees}, this implies that
$c_X(B_m)=O\left((\log m)^{1/q_X}\right)$. From Bourgain's lower
bound~\cite{Bou86-trees}, or alternatively
Claim~\ref{claim:lowertrees} below, we deduce that $q_X\le p$, as
required. The fact that $X$ is superreflexive follows from
Bourgain's characterization of
superreflexivity~\cite{Bou86-trees}.
\end{proof}

\begin{corollary} Let $X$ be a Banach lattice with
$\Pi_p(X)<\infty$. Then for every $q>p$, $X$ admits a $q$-convex
equivalent norm.
\end{corollary}
\begin{proof}
This is a direct consequence of a theorem of
Figiel~\cite{Figiel76} (see~\cite{LT79}, page 100) which says that
a Banach lattice with cotype $q$ and non-trivial type can be
renormed to be $q$-convex ($X$ has non-trivial type since it is
superreflexive).
\end{proof}

\remove{
We are left with the following interesting problem, a positive
solution of which would show that for Banach spaces Markov
$p$-convexity is equivalent to the existence of an equivalent norm
with modulus of convexity of power type $p$.

\medskip
\noindent{\bf Problem:} Let $X$ be a Banach space such that
$\Pi_p(X)<\infty$. Does this imply that $X$ admits an equivalent
norm which is $p$-convex?
}

\subsection{Distortion lower bounds}\label{sec:lower distortion}

We can now use the discrepancy between $\Pi_p(X)$ and $\Pi_p(Y)$ to prove
distortion lower bounds for embeddings between the two spaces.

\begin{lemma}\label{lem:mcdist}
Let $(X,d_X), (Y,d_Y)$ be metric spaces, then for every $p < \infty$, we have
$$
c_Y(X)\ge \frac{\Pi_p(X)}{\Pi_p(Y)}.
$$
\end{lemma}

\begin{proof}
Fix $\Pi>\Pi_p(Y)$. Let $g : X \to Y$ be a bi-Lipschitz map, let
$\{X_t\}_{t=0}^\infty$ be a Markov chain with state space
$\Omega$, and let $f : \Omega \to X$. Then
\begin{eqnarray*}
&&\!\!\!\!\!\!\!\!\!\!\!\!\!\!\!\sum_{k=0}^m\sum_{t=1}^{2^{m}}\frac{\E\left[d_X\left(f\left(X_{t}\right),f(\widetilde
X_{t}(t-2^k))\right)^p\right]}{2^{kp}}\\ &\leq&
\|g^{-1}\|_\Lip^p\cdot
\sum_{k=0}^m\sum_{t=1}^{2^{m}}\frac{\E\left[d_Y\left(g(f\left(X_{t}\right)),g(f(\widetilde
X_{t}(t-2^k)))\right)^p\right]}{2^{kp}} \\
&\leq & \|g^{-1}\|_\Lip^p \cdot \Pi^p\cdot
\sum_{t=1}^{2^m}\E [d_Y(g(f(X_t)),g(f(X_{t-1})))^p] \\
&\leq & \|g\|_\Lip^p \cdot \|g^{-1}\|_\Lip^p \cdot \Pi^p\cdot
\sum_{t=1}^{2^m}\E [d_Y(f(X_t),f(X_{t-1}))^p].
\end{eqnarray*}
It follows that $\Pi_p(X) \leq c_Y(X) \cdot \Pi_p(Y)$, as
required.
\end{proof}

As a warm up to the more involved lower bounds that will follow,
we show how Markov convexity can be used to prove Bourgain's
theorem for complete binary trees.

\begin{claim}\label{claim:lowertrees}
For every $m \in \mathbb N$, we have $\Pi_p(B_{2^m}) \geq
2^{1-\frac{2}{p}}\cdot m^{\frac1p}.$
\end{claim}

\begin{proof}
Let $\{X_t\}_{t=0}^\infty$ be the forward random walk on $B_{2^m}$
(which goes left/right with probability $\frac12$), starting from
the root, with the leaves as absorbing states. Then
$$
\sum_{t=1}^{2^m} \mathbb E\left[d_{B_{2^m}}(X_t,X_{t-1})^p\right]
\leq 2^m.
$$
Moreover, in the forward random walk, after splitting at time $r \leq 2^m$
with probability at least $\frac12$ two independent walks will
accumulate distance which is at least twice the number of steps (until
a leaf is encountered). Thus
$$
\sum_{k=0}^m\sum_{t=1}^{2^{m}}\frac{\E\left[d_{B_{2^m}}\left(X_{t},\widetilde
X_{t}\left(t-2^{k}\right)\right)^p\right]}{2^{kp}} \geq
\sum_{k=0}^m \sum_{t=1}^{2^m-2^k} \frac{1}{2^{kp}} \cdot \frac12
\cdot 2^{(k+1)p} \geq 2^{p-2}\cdot m \cdot 2^{m}.
$$
The claim follows.
\end{proof}

Since $L_p$ is Markov $\max\{2,p\}$-convex for every $p > 1$,
combining Claim~\ref{claim:lowertrees} with Lemma \ref{lem:mcdist}
recovers Bourgain's result~\cite{Bou86-trees}, i.e. for every $p >
1$, we have $c_p(B_k) \geq \Omega\left((\log k)^{\min\{\frac12,
\frac1p\}}\right)$.

\medskip
\noindent {\bf Simple random walks with positive speed.} In fact,
the proof of Claim~\ref{claim:lowertrees} applies in more general
situations where a random walk has positive speed. We consider
some examples.

Let $G = (V,E)$ be an infinite, vertex-transitive graph of bounded
degree. Let $\{X_t\}_{t=0}^\infty$ be a simple random walk on $G$
starting from an arbitrary vertex. Denote by $d_G$ the shortest
path metric on $G$. One defines the {\em speed of the random walk}
as the limit
$$
\lim_{t\to \infty} \frac{\mathbb E\,d_G(X_0, X_t)}{t}.
$$
Subadditivity implies that the limit above always exists.
\begin{lemma}\label{lem:speed}
If the speed of the simple random walk on a vertex-transitive
graph $G$ is at least $s > 0$, then $\Pi_p(B_G(R)) =
\Omega\left((\log R)^{1/p}\right)$, where $B_G(R)$ denotes the
ball of radius $R$ in $G$.  In particular for every $p > 1$,
$$
c_p(B_G(R)) = \Omega\left((\log R)^{\min\{\frac12,
\frac1p\}}\right).
$$
\end{lemma}

\begin{proof}
The proof is similar to that of Claim \ref{claim:lowertrees}.
One simply observes that for two independent simple random
walks $X_t, \widetilde X_t$ started at the same point, we have
$$
\lim_{t \to \infty} \frac{\mathbb E\,d_G(X_t, \widetilde X_t)}t =
\lim_{t\to \infty}\frac{\mathbb E\,d_G(X_0, X_{2t})}{t} \geq 2s >
0.
$$
In particular, for $t$ large enough, with constant probability
we have $d_G(X_t, \widetilde X_t) = \Omega(t)$.
\end{proof}

As an application, consider the lamplighter group over $\mathbb Z^d$.
This is a group with elements $(f,x)$, where $x \in \mathbb Z^d$,
and $f : \mathbb Z^d \to \{0,1\}$ with $f(y) = 0$ for all but
finitely many $y \in \mathbb Z^d$.  Traditionally, one imagines
a lamp placed at every element of $\mathbb Z^d$, where each lamp
can either be on or off.  In the pair $(f,x)$, $f$ denotes the
settings of all the lamps, and $x$ denotes the position of the
lamplighter.  Accordingly, the generating set consists of two types
of moves.
\begin{enumerate}
\item The lamplighter can move to an adjacent vertex in $\mathbb Z^d$, i.e.
$(f,x) \mapsto (f,x')$ where $x'$ is adjacent to $x$
in the standard Cayley graph of $\mathbb Z^d$ or
\item The lamplighter can turn on/off the lamp at $x$, i.e.
$(f,x) \mapsto (f',x)$ where $f'(y) = f(y)$ for $y \neq x$
and $f'(x) = 1-f(x)$.
\end{enumerate}
We will use $L(\mathbb Z^d)$ to denote the associated group as
well as the Cayley graph with the described generators. A result
of Ka{\u\i}manovich and Vershik \cite{KV79} shows that the simple
random walk on $L(\mathbb Z^d)$ has positive speed for $d > 2$.
Using Remark \ref{rem:superreflexive} and Lemma \ref{lem:speed},
we conclude:

\begin{corollary}
For $d > 2$, the word metric on $L(\mathbb Z^d)$ does not embed into any $p$-convex Banach space.
In particular, if $B_{L(\mathbb Z^d)}(R)$ denotes a ball of radius $R$, then for $p > 1$,
$$
c_p\left(B_{L(\mathbb Z^d)}(R)\right) \geq \Omega\left((\log
R)^{\min\{\frac12,\frac1p\}}\right).
$$
\end{corollary}

We remark that, by a theorem of Varopoulos~\cite{Var85}, the
simple random walk on the Cayley graph of a finitely-generated
group has positive speed if and only if there exists a bounded,
non-constant, harmonic function on the graph.

Finally, we consider the finite lamplighter groups over $\mathbb Z_N = \mathbb Z/(N \mathbb Z),$
which we denote by $L(\mathbb Z_N)$.  In this case, the simple random walk on
$L(\mathbb Z_N)$ does not have positive speed, but it is still possible to
prove a distortion lower bound because the Markov chains in the definition
of Markov convexity need not be reversible.  In particular,
consider the chain $\{X_t\}_{t=0}^\infty$ defined as follows.
$X_0 = (f,0)$ where $f \equiv 0$, i.e. all lamps are turned off.
If $X_t = (f, i)$, then with probability $\frac12$, we put
$X_{t+1} = (f, i+1)$, and with probability $\frac12$ we put
$X_{t+1} = (f',i+1)$ where $f'(i+1) = 1-f(i+1)$.  Arguing
essentially exactly as in Claim \ref{claim:lowertrees} for
times $t \leq N$, we have the following.

\begin{proposition}\label{prop:lampZ2}
For every $p < \infty$, we have $\Pi_p(L(\mathbb Z_N)) \geq
\Omega\left((\log N)^{\frac1p}\right)$. In particular,
$c_p(L(\mathbb Z_N)) \geq \Omega\left((\log
N)^{\min\{\frac12,\frac1p\}}\right)$.
\end{proposition}
Proposition~\ref{prop:lampZ2} can also be proved by exhibiting an
embedding of a complete binary tree of depth $\Theta(N)$ into
$L(\Z_N)$---see~\cite{LPP96}.

\subsection{Weak prototypes and Markov convexity}\label{sec:def weak}

In this section we study the Markov convexity properties of a
special class of trees called weak prototypes. These trees will
play a central role in Section~\ref{sec:dist formula}, where
Theorem~\ref{thm:formula} is proved. We begin with some
definitions (we continue using here the notation of
Section~\ref{section:coloring}).

  In what follows, by a path metric
$P=(u_1,\ldots,u_m)$ we simply mean a graph theoretical path from
$u_1$ to $u_m$ with edges
$(u_1,u_2),(u_2,u_3),\ldots,(u_{m-1},u_m)$ and edge weights
$\{\ell(u_j,u_{j+1})\}_{j=1}^{m-1}\subseteq [0,\infty)$. The
length of $P$, denoted $\ell(P)$, is given by
$\ell(P)=\sum_{j=1}^{m-1} \ell(u_j,u_{j+1})$. Given a monotone
path $P$ in $T$, and a set of vertices $(v_1,\ldots,v_m)$ on $P$,
ordered down the tree and not necessarily containing all the
vertices of $T$ lying on $P$, we will call the path metric on
$(v_1,\ldots,v_m)$ with the edge weights
$\{d_T(v_j,v_{j+1})\}_{j=1}^{m-1}$, the path metric induced by $T$
on $(v_1,\ldots,v_m)$.

Given a path metric $P=(u_1,\ldots,u_m)$ and $\e,\delta\in (0,1)$
we shall say the path $P$ is $(\e,\delta)$-weak if at least an
$\e$-fraction of the length of $P$ is composed of edges of length
at most $\delta \ell(P)$, i.e.
$$
\sum_{\substack{j\in \{1,\ldots,m-1\}\\ \ell(u_j,u_{j+1})\le
\delta \ell(P)}}\ell(u_j,u_{j+1})\ge \e
\ell(P)=\e\sum_{j=1}^{m-1}\ell(u_j,u_{j+1}).
$$

A monotone path $P(u,v)$ in $T$ will be called {\em degree-$2$
$(\e,\delta)$-weak} if the following condition holds true. Let
$(u_1,\ldots,u_m)$ be the vertices of $T$ on $P$, ordered down the
tree, who have at least two children in $T$. Then we require that
the path metric induced by $T$ on $(u,u_1,\ldots,u_m,v)$ is
$(\e,\delta)$-weak. In other words, call a monotone path $P$ in
$T$ a {\em strait} if every vertex on $P$ has exactly one child,
except possibly for the initial and final vertices. Then $P(u,v)$
is degree-$2$ $(\e,\delta)$-weak if at least an $\e$-fraction of
the length of $P(u,v)$ is composed of maximal straits of length at
most $\delta d_T(u,v)$.

\begin{definition}\label{def:weak prototype} Fix $\e,\delta,R>0$. A tree $T=(V,E)$ with
edge lengths $\ell:E\to (0,\infty)$ is called an
$(\e,\delta)$-weak prototype with height ratio $R$ if the
following conditions are satisfied.
\begin{itemize}
\item Every non-leaf vertex of $T$ has exactly one or two children.
\item Every root-leaf path of $T$ is degree-$2$ $(\e,\delta)$-weak.
\item If $h$ is the length of the shortest root-leaf path in $T$ and
$h'$ is the length of the longest root-leaf path in $T$, then
$h'/h\le R$.
\end{itemize}
\end{definition}

\subsubsection{Markov convexity for unweighted weak prototypes}

First, we will prove a lower bound on the
Markov convexity constants
of a special class of unweighted weak
prototypes.
Later, we will show that every
weak prototype can be approximated by a weak prototype
satisfying these conditions.

\begin{theorem}\label{thm:unweightedweak}
Let $(T,d_T)$ be an {\em unweighted} $(\varepsilon,\delta)$-weak prototype
with height ratio 1 and height $2^m$ for some $m \in \mathbb N$.
Then for every $p \geq 1$, we have
$$
\Pi_p(T) \geq \left(\frac{\varepsilon}{4} \left[\log_2
(\varepsilon/\delta)-4\right]\right)^{1/p}.
$$
\end{theorem}

\begin{proof}
Let $r$ be the root of $T$.
Let $\{X_t\}_{t=0}^{\infty}$ be the Markov chain on $T$
defined as follows.  Initially, $X_0 = r$.
If $X_t$ is a leaf node, then $X_{t+1} = X_t$, and otherwise $X_{t+1}$
is a uniformly random child of $X_t$.

First, we have $d_T(X_{t-1}, X_t) \leq 1$ for every $t \geq 1$.
Thus it suffices to show that
$$
\frac{1}{2^m}  \sum_{k=0}^m \sum_{t=1}^{2^m} \frac{\mathbb
E\left[d_T\left(X_t, \widetilde
X_t(t-2^k)\right)\right]^p}{2^{kp}} \geq \frac{\varepsilon}{4}
\left[ \log_2 \left(\frac{\varepsilon}{\delta}\right)-4\right].
$$

Recall that a monotone path $P$ in $T$ is a {\em strait} if
every node of $P$ has exactly one child, except possibly for
the initial and final nodes.
Additionally, say that a node $v \in T$ is a {\em branch point} if $v$ has at least two children.
Clearly the edges
of every root-leaf path partition
into maximal straights with branch points
at the ends (except for the root and leaves).
Let $\mathcal B_k(t)$ be the event that the set $\{X_t, X_{t+1}, \ldots, X_{t + 2^{k-1}}\}$
contains a branch point.
Observe that whenever $2^{k-1} \geq \delta 2^m$, we have
\begin{eqnarray}\label{eq:sumprobs}
\sum_{t=0}^{2^m-1} \Pr[\mathcal B_k(t)] \geq \sum_{t=0}^{2^m-1}
\Pr\left[\textrm{$X_t$ falls in a maximal strait
of length at most $2^{k-1}$}\right] \geq \varepsilon 2^m,
\end{eqnarray}
since every root-leaf path of $T$ is degree-2 $(\varepsilon,\delta)$-weak.
Furthermore, if $k \leq m$, and $t \leq 2^m - 2^{k}$, then
\begin{equation}\label{eq:eventB}
\textrm{$\mathcal B_k(t)$ occurs} \implies \Pr\left[d_T(X_{t+2^k},
\widetilde X_{t+2^k}(t)) \geq 2^{k}\right] \geq \frac12,
\end{equation}
since upon hitting a branch point, the two chains will diverge with probability at least $\frac12$
for at least $2^{k-1}$ additional steps.

We conclude that when $2^{k-1} \geq \delta 2^m$ and $k \leq m$,
\begin{eqnarray}
\sum_{t=1}^{2^m} \mathbb E\left[d\left(X_t, \widetilde
X_t(t-2^k)\right)\right]^{p} \nonumber
&\geq & \sum_{t=0}^{2^m - 2^k} \mathbb E\left[d\left(X_{t+2^k}, \widetilde X_{t+2^k}(t)\right)\right]^p \\
&\geq & \sum_{t=0}^{2^m-2^k} \frac12 \cdot 2^{kp} \cdot \Pr[\mathcal B_k(t)] \label{eq:exec1} \\
&\geq & 2^{kp-1} \cdot \left( \varepsilon 2^m - 2^k \right), \label{eq:exec2}
\end{eqnarray}
where in \eqref{eq:exec1} we used \eqref{eq:eventB}, and in \eqref{eq:exec2} we used \eqref{eq:sumprobs}
along with a correction term for boundary values of $k$.  Therefore,
\begin{eqnarray*}
\frac{1}{2^m}  \sum_{k=0}^m \sum_{t=1}^{2^m} \frac{\mathbb
E\left[d_T\left(X_t, \widetilde
X_t(t-2^k)\right)\right]^p}{2^{kp}}
&\geq& 2^{-(m+1)} \sum_{k \geq 1+ \log_2 (\delta 2^m)}^{m} \max \left\{0, \varepsilon 2^m - 2^k\right\} \\
&\geq & \frac{\varepsilon}{4} \left[\log\left(\frac{\varepsilon}{\delta}\right) - 4\right].
\end{eqnarray*}
The proof of Theorem~\ref{thm:unweightedweak} is complete.
\end{proof}

\subsubsection{Distortion bounds for weak prototypes}

In this section, we show how pass from a finite tree $T$ to a more
well-behaved tree $\widetilde T$ such that $c_p(\widetilde T) =
O(1)\cdot c_p(T)$ for every $p \in [1,\infty)$. We use this
transformation to prove distortion lower bounds for arbitrary weak
prototypes.

\begin{lemma}\label{lem:rtree}
Let $(T,d_T)$ be a finite, graph-theoretic metric tree, and let
$T_\mathbb R$ be the $\mathbb R$-tree that results from replacing
every edge of $e \in E(T)$ by a closed interval whose length is
$\mathsf{length}(e)$. Then for every $p \in [1,\infty)$, we have
$c_p(T_\mathbb R) \leq 5 c_p(T)$.
\end{lemma}

\begin{proof}
Fix a root $r$ of $T$ (and, in particular, an orientation of the
edges). Let $f : T \to L_p$ be an embedding of $T$. Let
$\{\beta_{uv}\}_{uv \in E(T)} \subseteq L_p$ be a system of
disjointly supported unit vectors each of whose support is also
disjoint from the support of $\mathrm{Im}(f)$. Denote a point $x
\in T_\mathbb R$ by $x = (u,v,\eta)$, where $uv \in E(T)$, we have
$d_T(r,u) \leq d_T(r,x) \leq d_T(r,v)$, and $d_T(x,u) = \eta \cdot
d_T(u,v)$ for $\eta \in [0,1]$. Assume that $\|f\|_\Lip = 1$. We
define an embedding $g : T_\mathbb R \to L_p$ by
$$g\left(u,v,\eta\right) = (1-\eta) f(u) + \eta f(v) + \eta\, d_T(u,v) \beta_{uv}.$$
Fix $(u,v,\eta), (u',v',\eta') \in T_\mathbb R$. If $u'$ is not a
descendant of $u$ or vice-versa then
$$
d_{T_\R}\left((u,v,\eta),(u',v',\eta')\right)=\eta
d_T(u,v)+\eta'd_T(u',v')+d_T(u,u').
$$
Thus
\begin{eqnarray*}
 \|&&\!\!\!\!\!\!\!\!\!\!\!\!\!\!\! g(u,v,\eta) - g(u',v',\eta')\|_p \\
&\leq & \|g(u,v,\eta) - g(u,v,0)\|_p + \|g(u,v,0)-g(u',v',0)\|_p
+ \|g(u',v',0)-g(u',v',\eta')\|_p \\
&=& \left\|\eta(f(v)-f(u))+\eta
d_T(u,v)\beta_{uv}\right\|_p+\left\|(1-\eta)(f(u)-f(u'))\right\|_p\\&\phantom{\le}&+
\left\|\eta'(f(v')-f(u'))+\eta'
d_T(u',v')\beta_{u'v'}\right\|_p\\&=& \eta\left(\|f(u)-f(v)\|_p^p+d_T(u,v)^p\right)^{1/p}+(1-\eta)\|f(u)-f(u')\|_p
\\&\phantom{\le}&+\eta'\left(\|f(u')-f(v')\|_p^p+d_T(u',v')^p\right)^{1/p}\\
&\le& 2^{1/p}\eta d_T(u,v)+2^{1/p}\eta' d_T(u',v')+
(1-\eta)d_T(u,u')\\
&\leq& 2^{1/p} \cdot d_{T_{\mathbb
R}}\left((u,v,\eta),(u',v',\eta')\right).
\end{eqnarray*}
If $u'$ is a strict descendant of $u$ then
\begin{eqnarray}\label{eq:forgotten case}
d_{T_\R}\left((u,v,\eta),(u',v',\eta')\right)=(1-\eta)
d_T(u,v)+\eta'd_T(u',v')+d_T(v,u'),
\end{eqnarray}
and a similar reasoning shows that $ \|g(u,v,\eta) -
g(u',v',\eta')\|_p\le 2^{1/p}\cdot d_{T_{\mathbb
R}}\left((u,v,\eta),(u',v',\eta')\right)$. The case of $u=u'$ is
even simpler, so we have shown that $\|g\|_{\Lip}\le
2^{1/p}\|f\|_{\Lip}$.

On the other hand, we will now show that $\|g^{-1}\|_{\Lip}\le
\frac{5}{2^{1/p}}$. Assume first of all that $u'$ is not a
descendant of $u$ or vice-versa. Then
\begin{eqnarray}
 \|&&\!\!\!\!\!\!\!\!\!\!\!\!\!\!\! g(u,v,\eta) - g(u',v',\eta')\|_p^p\nonumber \\
&\geq & [\eta\, d_T(u,v)]^p + [\eta' d_T(u',v')]^p  +
\left|\|f(u) - f(u')\|_p - \eta\|f(u) -  f(v)\|_p - \eta'\|f(u') -  f(v')\|_p\right|^p \nonumber\\
&\geq & [\eta\, d_T(u,v)]^p + [\eta' d_T(u',v')]^p  +
\left|\frac{d_T(u,u')}{\|f^{-1}\|_\Lip} - \eta\,d_T(u,v) - \eta'd_T(u',v')\right|^p \nonumber\\
&\geq & \frac{5}{2} \left[ \frac{1}{5} \left(\frac{d_T(u,u')}{\|f^{-1}\|_\Lip} + \eta\,d_T(u,v) + \eta'd_T(u',v')\right)\right]^p\label{eq:2/5}\\
&\geq & 2 \left[ \frac{1}{5} \left(\frac{d_{T_{\mathbb
R}}\left((u,v,\eta),(u',v',\eta')\right)}{\|f^{-1}\|_\Lip}\right)\right]^p\nonumber,
\end{eqnarray}
Where in~\eqref{eq:2/5} we used the convexity of the function
$a\mapsto |a|^p$, which implies that for all $a,b,c\in \R$ we have
$|a|^p+|b|^p+|c|^p\ge \frac{5}{2}\left|\frac25 a+\frac25 b+\frac15
c \right|^p$.

If $u'$ is a strict descendant of $u$ then
$d_{T_\R}\left((u,v,\eta),(u',v',\eta')\right)$ is given
by~\eqref{eq:forgotten case}. Denote this distance by $D$, and for
the sake of simplicity write $L=\|f^{-1}\|_{\Lip}$. Since
$$
 \|g(u,v,\eta) - g(u',v',\eta')\|_p^p
\geq  [\eta\, d_T(u,v)]^p + [\eta' d_T(u',v')]^p\ge
2^{1-p}\left[\eta\, d_T(u,v)+\eta' d_T(u',v')\right]^p,
$$
we may assume that $\eta\, d_T(u,v)+\eta' d_T(u',v')<
\frac{2D}{5L}$. In this case
\begin{eqnarray*}
 \|g(u,v,\eta) - g(u',v',\eta')\|_p&\ge& \left\|(1-\eta)f(u)+\eta f(v)-(1-\eta')f(u')-\eta'
 f(v')\right\|_p\\
 &\ge& \|f(u)-f(u')\|_p-\eta
 \|f(u)-f(v)\|_p-\eta'\|f(u')-f(v')\|_p\\
 &\ge& \frac{d_T(u,u')}{L}-\eta d_T(u,v)-\eta'
 d_T(u',v')\\&\ge&
\frac{D-\eta' d_T(u',v')}{L}-\frac{2D}{5L}\\
&\ge& \frac{D}{L}-\frac{D}{5L^2}-\frac{2D}{5L}\\&\ge&
\frac{D}{5L}.
\end{eqnarray*}
The remaining case is when $u=u'$ and $v=v'$. But then $
\|g(u,v,\eta) - g(u',v',\eta')\|_p=|\eta-\eta'|\cdot
\|f(u)-f(v)\|_{p}$, and the required lower bound is trivial.

We have thus proved that $\|g\|_\Lip \cdot \|g^{-1}\|_\Lip \leq 5
\|f\|_\Lip \cdot \|f^{-1}\|_\Lip$, as required.
\end{proof}

\begin{remark}{\em
The above lemma does {\em not} hold if we allow ``Steiner'' nodes
in the tree $T$.  To observe this, consider the subset $L
\subseteq B_m$ of leaves of a complete binary tree of height $m$,
and let $r$ be the root of $B_m$. Then it is not difficult
 to see that $c_2(L \cup \{r\}) \leq O(1)$ (independent of $m$), while
$c_2(B_m) \to \infty$ by Bourgain's theorem for $B_m$
\cite{Bou86-trees}.}
\end{remark}

We now replace any weak prototype by an ``equivalent'' prototype with height ratio 1.

\begin{lemma}\label{lem:idealtrans}
Let $(T,d_T)$ be any finite metric tree.  Then there exists a
finite, unweighted metric tree $(\widetilde T, d_{\widetilde T})$
with height $2^m$ for some $m \in \mathbb N$ such that
$c_p(\widetilde T) \leq O(1) \cdot c_p(T)$ for any $p \in
[1,\infty)$. Furthermore, if $T$ is an $(\varepsilon,\delta)$-weak
prototype with height ratio $R$, then $\widetilde T$ is an
unweighted $\left(\frac{\varepsilon}{2R},\delta\right)$-weak
prototype with height ratio 1.
\end{lemma}

\begin{proof}
Fix a root $r$ of $T$. Since $T$ is finite, by rescaling and
paying an arbitrarily small distortion, we may assume that all
edge lengths are integral.  For every node $x \in T$, Let $m =
\lceil \log_2 \max_{x \in T} d_T(x,r) \rceil$. We now define a
tree $T'$ as follows. For every leaf $\ell \in T$, define a new
node $\widetilde \ell$, and create a new edge $(\ell,\widetilde
\ell)$ of length $2^m - d_T(r,\ell)$. Thus the length of every
root-leaf path in $T'$ is exactly $2^m$. To see that $c_p(T') =
\Theta(1) \cdot c_p(T)$, let $f : T \to L_p$ be an embedding of
$T$, and let $\{\beta_\ell\}\subseteq L_p$ be a system of
disjointly supported vectors each of whose support is disjoint to
$\mathrm{Im}(f)$. One can extend the embedding by defining
$f(\widetilde \ell) = f(\ell) + d_{T'}(\ell,\widetilde \ell) \cdot
\beta_\ell$ so that $c_p(T') \leq O(1)\cdot c_p(T)$. Observe that
if $T$ had height ratio $R$, then the length of any root-leaf path
in $T'$ has increased by at most a factor $2R$ over its previous
length in $T$.

We pass from $T'$ to $T'_{\mathbb R}$ using Lemma \ref{lem:rtree},
and then to $\widetilde T$ by simply taking the vertex set of
$\widetilde T$ to be $V(\widetilde T) = \{ v \in T' : d_{T'}(v,r)
\in \mathbb N \}$. We define $d_{\widetilde T}$ as the {\em
unweighted} shortest path metric on $T$.  then $(\widetilde T,
d_{\widetilde T})$ embeds isometrically into $T'_{\mathbb R}$.
Hence $c_p(\widetilde T) = \Theta(1) \cdot c_p(T'_{\mathbb R}) =
\Theta(1) \cdot c_p(T') = \Theta(1)\cdot c_p(T)$. Furthermore,
every root-leaf path in $\widetilde T$ has length precisely $2^m$.
Finally, observe that if $T$ was $(\varepsilon,\delta)$-weak with
height ratio $R$, then $\widetilde T$ is an unweighted
$(\varepsilon/(2R), \delta)$-weak prototype (because some
root-leaf path from $T$ might have increased by a factor of at
most $2R$) with height ratio $1$.
\end{proof}

The following corollary follows from
Theorem~\ref{thm:unweightedweak} and Lemma~\ref{lem:idealtrans}.

\begin{corollary}\label{coro:weak lower}
Let $(T,d_T)$ be an $(\varepsilon,\delta)$-weak prototype with height ratio $R$,
then for any $p > 1$,
$$
c_p(T,d_T)\geq \Omega(1) \cdot c_p(\widetilde T, d_{\widetilde T})
\geq \Omega(1) \cdot \Pi_q(\widetilde T, d_{\widetilde T}) \geq
\Omega(1) \cdot \left(\frac{\varepsilon}{R}
\log\left(\frac{\varepsilon}{\delta R}\right)\right)^{1/q},
$$
where $q = \max\{2,p\}$ and $\widetilde T$ is the associated
unweighted prototype from Lemma \ref{lem:idealtrans}.
\end{corollary}

The corollary follows by applying Theorem \ref{thm:unweightedweak}
to $\widetilde T$ and using the relationship between Markov
convexity and distortion from Lemma \ref{lem:mcdist}, along with
the known Markov convexity of $L_p$ spaces (Remark
\ref{rem:superreflexive}).

\subsubsection{The Cantor trees}\label{sec:cantor}

Recall that in Theorem \ref{thm:upper-binary}, we showed that
for any tree $T$ and $p \geq 1$, we have, for every $c > 1$,
$$
c_p(T)\le O(1)\left(\frac{c}{c-1}\cdot
\mathscr{B}_T(c)\right)^{\min\left\{\frac{1}{p},\frac12\right\}}.
$$
Here, we will show that this dependence on $\mathscr{B}_T(c)$
cannot be improved by exhibiting a family $\{C_i\}_{i=0}^\infty$
of metric trees with $|C_i| \to \infty$ and such that for any
fixed $c
> 1$,
\begin{equation}\label{eq:cantor}
c_p(C_i) \geq \Omega(1) \cdot \Pi_{\max\{2,p\}}(C_i) \geq
\Omega(1) \cdot \mathscr{B}_{C_i}(c)^{\min\left\{\frac1p,\frac12\right\}}
\end{equation}

Let $T$ be a rooted (unweighted) graph-theoretic tree.  For any
root leaf path $P = \{v_0, v_1, \ldots, v_m\}$, we define the {\em
downward degree sequence} $d_{\downarrow}(P) =
\{d_\downarrow(v_0), d_\downarrow(v_1), \ldots,
d_\downarrow(v_m)\}$ where $d_\downarrow(v)$ is the number of
children of $v$ in $T$. We will say that $T$ is a {\em spherically
symmetric tree (SST)} if for any pair of root-leaf paths $P,P'$ we
have $d_\downarrow(P) = d_\downarrow(P')$. Clearly any such tree
can be completely specified by giving the degree sequence of a
root-leaf path (see Figure \ref{fig:sst}).

\begin{figure}
 \centerline{\hbox{
        \psfig{figure=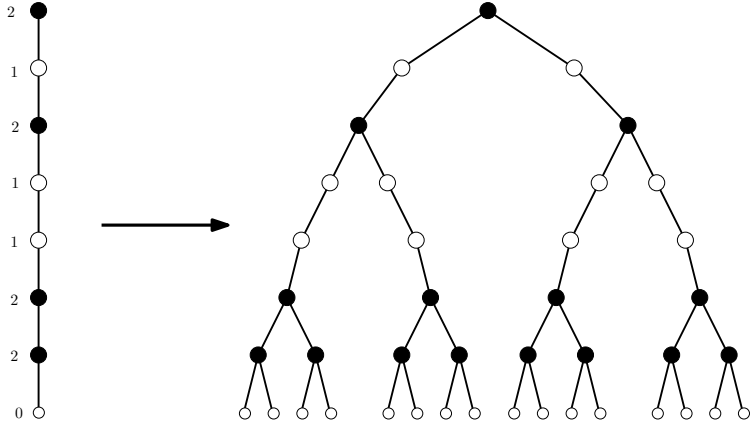,height=2.2in}
        }}
        \caption{A downward degree sequence and the corresponding SST.}\label{fig:sst}
\end{figure}

\medskip
\noindent
{\bf Definition of the Cantor trees.}
We now describe a family of downward degree sequences inductively.
For two sequences $S,S'$ we define $S \otimes S'$
as their concatenation.  For every $i \in \mathbb N$,
we use $\mathsf{ones}(i) = \otimes^i \{1\}$ to denote a sequence of $i$ ones.
Now define $S_0 = \{2\}$ and inductively
$$S_{i+1} = S_i \otimes \mathsf{ones}(2^i-1) \otimes S_i.$$
Hence the first few sequences are $\left\{\{2\}, \{22\},
\{22\,1\,22\}, \{22122\,111\,22122\},\ldots\right\}$. To make
these proper downward degree sequences, we define $\widetilde S_i$
to be $S_i$ except with the last element changed from $2$ to $0$.
Finally, we let $C_i$ be the unique SST with downward degree
sequence $\widetilde S_i$.  We call these {\em Cantor trees}
because the patterns of $2$'s resemble finite approximations to
the middle-thirds Cantor set. It is clear that
$\mathrm{length}(S_i) = 2 \cdot \mathrm{length}(S_{i-1}) + 2^{i-1}
- 1 = i\cdot 2^{i-1}+1,$ and that $\log \log |C_i| = \Theta(i)$.
The next two lemmas are somewhat less obvious.

\begin{lemma}
For every $i \geq 1$, the tree $C_i$ is a $(\tfrac12, 2^{-i/3})$-weak prototype.
\end{lemma}

\begin{proof}
We need to show that every root-leaf path in $C_i$ is degree-2 $(\tfrac12, 2^{-i/3})$-weak.
Fix any such path $P$.
It is easy to see that the maximal straits in $P$ are given by consecutive
sequences of $1$'s in the downward degree sequence of $C_i$:  A sequence of $k$
consecutive $1$'s refers to a strait of length $k+1$.
Therefore for every $j \leq i-1$, there are $2^{i-1-j}$ disjoint
maximal straits of length $2^j$ in $P$.
The question becomes how small we need to choose $m$ before
$$
\frac12 \,\mathrm{height}(C_i) = \frac12 \left(i\cdot 2^{i-1} + 2\right) \leq \sum_{j=m}^{i-1} 2^{i-1-j} \cdot 2^j = 2^{i-1} (i-m).
$$
Clearly we must have $m < i/2$, hence $C_i$ is a $(\tfrac12,\delta)$-weak prototype
for
$$
\delta = \frac{2^{i/2}}{i \cdot 2^{i-1} + 1} \leq 2^{-i/3}.
$$
\end{proof}

Combining this with Theorem \ref{thm:unweightedweak} yields the following.

\begin{corollary}
For every $p < \infty$, $\Pi_p(C_i) \geq
\Omega\left(i^{1/p}\right)$.
\end{corollary}

The following claim completes the proof of \eqref{eq:cantor}.

\begin{claim}\label{lem:nobin}
For every fixed $c > 1$, $\mathscr B_{C_i}(c) \leq O(i)$ as $i \to
\infty$.
\end{claim}

\begin{proof}
The idea of the proof is simple:  If the edges of $B_m$ are mapped
far apart in $C_i$, then we can use the diameter of $C_i$ to upper
bound the size of $m$.  Otherwise, if the edges are mapped close
together, then essentially the entire image of $B_m$ must lie
inside some copy of $C_{i-1}$ in $C_i$.  This is because there
is a ``buffer'' of length $2^{i-1}$ between copies of $C_{i-1}$ in $C_i$
which contains no branch points.  An edge of $B_m$ must stretch
over this buffer if the image of $B_m$ spans multiple
copies of $C_{i-1}$.

For our induction, it will be easier to bound
$\mathscr B_{\widetilde C_i}(c)$ for a slightly different
family of trees $\widetilde C_i$.  Let $\widetilde C_i$ be
the tree $C_i$ with the following two additions:
\begin{enumerate}
\item We append a path $H_i$ of length $2^{i-2}$ to the root of $C_i$.
\item We append a path of length $2^{i-2}$ to every leaf of $C_i$.
We will use $\mathcal L = \{L_i\}$ to refer to this family of paths.
\end{enumerate}
Clearly $\mathscr B_{C_i}(c) \leq \mathscr B_{\widetilde C_i}(c)$.

We may assume that $i \geq 1$ is sufficiently large with respect to $c$.
Let $f : B_m \to \widetilde C_i$ be a bi-Lipschitz
embedding of $B_m$ into $\widetilde C_i$ with distortion $c = \|f\|_\Lip \cdot \|f^{-1}\|_\Lip$.
Assume, for the sake
of contradiction, that $m \geq 256\,i\cdot c \log (c+1)$.

Clearly $$\mathrm{diam}(\widetilde C_i) \geq
\max_{u,v \in B_m} d_{\widetilde C_i}(f(u),f(v)) \geq \frac{2m}{\|f^{-1}\|_\Lip} \geq \frac{2m \|f\|_\Lip}{c}.$$
Since $\mathrm{diam}(\widetilde C_i) \leq i \cdot 2^{i+2}$, we conclude
that
\begin{equation}\label{eq:stretch}
\max_{uv \in E(B_m)} d_{\widetilde C_i}(f(u),f(v)) = \|f\|_\Lip \leq \frac{i \cdot 2^{i+2} \cdot c}{m} \leq \frac{2^{i-6}}{\log (c+1)},
\end{equation}
where $E(B_m)$ is
the set of edges in $B_m$.

We will now show that \eqref{eq:stretch} implies that $f(B_m)$ is contained
completely inside an isometric copy of $\widetilde C_{i-1}$.
By induction, this will be a contradiction and finish the proof.
Let us consider a ``top-down'' decomposition of $\widetilde C_i$ into disjoint pieces.
From the root downward, we see $H_i$, then a copy of $C_i$,
then the family of paths $\mathcal L$.  If we also break
$C_i$ into constituent pieces, we see:
\begin{enumerate}
\item $H_i$,
\item a copy $C_{i-1}$,
\item a family of paths $\mathcal P$ of length $2^{i-1}$ connected to the leaves of (2),
\item copies of $C_{i-1}$ connected to every endpoint of the paths from (3),
\item the family of paths $\mathcal L$ connected to the leaves of the copies of $C_i$ from (4).
\end{enumerate}

We now define a family of disjoint sub-trees of $\widetilde C_i$ each of which
is an isometric copy of $\widetilde C_{i-1}$.  The first copy $\widetilde C_{i-1}^{(0)}$
consists of the bottom $2^{i-3}$ nodes of $H_i$, the copy of $C_{i-1}$ from (2) above,
and the top $2^{i-3}$ nodes of each path $p \in \mathcal P$ (from (3)).
The other copies are indexed by paths $p \in \mathcal P$.  For each such path,
we construct $\widetilde C_{i-1}^{(p)}$ using the bottom $2^{i-3}$ nodes of $p$,
the copy of $C_{i-1}$ from (4) connected to the bottom of $p$, and the top $2^{i-3}$ nodes
of each path from $\mathcal L$ connected to this copy of $C_{i-1}$.

We claim that there exists some $j \in \{0\} \cup \mathcal P$ for which
$f(B_m) \subseteq \widetilde C_{i-1}^{(j)}.$  We now prove the most difficult
case; the other cases are similar.  Suppose, for the sake
of contradiction, there exist $x,y \in B_m$ for
which $f(x) \in \widetilde C_{i-1}^{(0)}$ and $f(y) \in \widetilde C_{i-1}^{(p)}$
for some $p \in \mathcal P$.  By \eqref{eq:stretch}, every edge of $B_m$
has length at most $2^{i-6}$, hence
there must be some node $z \in B_m$ for which $f(z)$
lies in the middle $2^{i-3}$ nodes of $p$.  In particular,
$|B_{\widetilde C_i}(f(z), r)| \leq 2r+1$ for every $r \leq 2^{i-3}$,
since $B_{\widetilde C_i}(f(z), 2^{i-3}) \subseteq p$.
Furthermore, $|f(B_m) \cap B_{\widetilde C_i}(f(z), r)| \leq (2r+1)\|f^{-1}\|_\Lip$.
On the other hand, $|B_{B_m}(z, r')| \geq 2^{r'/2}$ for $r' \leq m$.

But we have
$$f\left(B_{B_m}\left(z,\frac{r}{\|f\|_\Lip}\right)\right) \subseteq B_{\widetilde C_i}(f(z), r).$$
Let $r = \min \{ 2i c\,\|f\|_\Lip, 2^{i-3}\}$.
Since, in particular $r \leq m \|f\|_\Lip$,
the above considerations yield
\begin{equation}\label{eq:conclude}
2^{\frac{r}{2\|f\|_\Lip}} \leq (2r+1)\|f^{-1}\|_\Lip = \frac{(2r+1)c}{\|f\|_\Lip} \leq \frac{4rc}{\|f\|_\Lip}.
\end{equation}
Observe that the inequality $2^{B} \geq 8Bc$ hold as long as $B \geq 10 \log (c+1)$ and $c \geq 1$,
but it is easy to check that for $i \geq 10$, we have $\frac{r}{2\|f\|_\Lip} \geq 10\log(c+1)$ (using
\eqref{eq:stretch}), yielding
a contradiction.  This completes the proof.
\end{proof}

\begin{remark}
\label{rem:rtree}
Observe that the two point space $A = \{x,y\}$ with, say, $d(x,y)=1$
is a tree metric for which $\Pi_p(A) = O(1)$ for every $p < \infty$.
On the other hand, it is easy to see that $\Pi_p([0,1]) = \infty$
for every $p < 2$, thus in general $\Pi_p(T_\mathbb R) \not\approx \Pi_p(T)$
for $p < 2$.  For $p \geq 2$, the relationship is less clear,
though we suspect that a similar phenomenon holds in this case.
A possible example for which $\Pi_2(T^{(i)}_\mathbb R) \not\approx \Pi_2(T^{(i)})$ is
when $T^{(i)}$ is the Cantor tree $C_i$ with every maximal strait replaced
by a single long edge.  Using techniques similar to Lemma \ref{lem:nobin},
one might show that $\Pi_2(T^{(i)}) = O(1)$ as $i \to \infty$
while $\Pi_2(T^{(i)}_\mathbb R) \approx \Pi_2(C_i) \to \infty$.
We do not pursue this line of reasoning further in the present work.
\end{remark}

\section{Characterizing the distortion: strong colorings and Markov
convexity}\label{sec:dist formula}

In this section we will continue to use the notation of
Section~\ref{section:coloring}. Moreover, unless explicitly stated
otherwise, all paths will be assumed to be monotone. Many of the
concepts and definitions used in this section were introduced in
Section~\ref{sec:def weak}, so we suggest that the reader will be
familiar with Section~\ref{sec:def weak} before reading the
present section.

The following result, which contains Theorem~\ref{thm:formula}, is
the main theorem of this section:

\begin{theorem}[The $L_p$ distortion of trees]\label{thm:Markov main}
For $1<p<\infty$ and every metric tree $T=(V,E)$,
$$
c_p(T)=\Theta\left(\Pi_{\max\{p,2\}}(T_\R)\right)=\Theta\left(\left[\log\left(\frac{2}{\delta^*(T)}\right)\right]^{\min\left\{\frac{1}{p},\frac12\right\}}\right),
$$
where the implied constants may depend only on $p$.
\end{theorem}

Before proving Theorem~\ref{thm:Markov main} we make some
observations. By Lemma~\ref{lem:mcdist} for every $q>0$ and every
two metric spaces $(X,d_X)$ and $(Y,d_Y)$, $c_Y(X)\ge
\frac{\Pi_q(X)}{\Pi_q(Y)}$. Since $L_p$ is $\max\{p,2\}$ uniformly
convex, Remark~\ref{rem:superreflexive} implies that
$\Pi_{\max\{p,2\}}(L_p)<\infty$. This observation, together with
Theorem~\ref{thm:matousek} and Lemma~\ref{lem:rtree}, implies that
$$
\Omega\left(\Pi_{\max\{p,2\}}(T_\R)\right)\le \frac15 c_p(T_\R)\le
c_p(T)\le
O\left(\left[\log\left(\frac{2}{\delta^*(T)}\right)\right]^{\min\left\{\frac{1}{p},\frac12\right\}}\right).
$$

Thus, using Corollary~\ref{coro:weak lower}, the proof of
Theorem~\ref{thm:Markov main} will be complete if show that if a
metric tree  $T=(V,E)$ does not admit any $\delta$-strong coloring
then there exists a subtree of $T$ which is
$\left(\Omega(1),2\cdot \delta^{\Omega(1)}\right)$-weak prototype
with height ratio $O(1)$. It is clearly enough to prove this for
small enough $\delta$, so we assume in what follows that
$\delta<(140)^{-2880}$ (the proof below yields much better
constants, but we chose this rough bound to simplify the ensuing
exposition). The proof of this assertion is analogous to the
proof of Theorem~\ref{thm:upperB}, where ``strong'' colorings
replace ``good'' colorings, and weak prototypes
take the place of complete binary trees.  Since the structure
of a weak prototype is not as cleanly recursive as that of a complete
binary tree, there are some inevitable added complications.
The argument will be broken down into several
steps.

\subsection{Preliminary results on paths in trees}

 In what
follows, given $u,v\in V$ we shall say that a set of consecutive
edges $C\subseteq P(u,v)$ is a $\delta$-cluster if $\ell(e)\le
\delta d_T(u,v)$ for every $e\in C$.

\begin{lemma}\label{lem:cluster} Fix $\alpha\in (0,\frac12)$, $\delta\in (0,1)$, and denote
$\tau=\frac{1}{2-4\alpha}$. Assume that $u,v\in V$ are such that
the path $P(u,v)$ is $\left(\frac12+\alpha,\delta\right)$-weak.
Then at least an $\alpha$-fraction of the length of $P(u,v)$ is
covered by $\delta$-clusters of length at least $\tau\delta
d_T(u,v)$. Moreover, at least an $\alpha$-fraction of the length
of $P(u,v)$ is covered by edge-disjoint $\delta$-clusters of
length between $\tau\delta d_T(u,v)$ and $(2\tau+1)\delta
d_T(u,v)$.
\end{lemma}

\begin{proof} Fix $u,v\in V$ and denote $P=P(u,v)$ and
$d=d_T(u,v)=\ell(P)$. Let $M$  be the set of maximal
$\delta$-clusters (with respect to inclusion)  contained in $P$.
In what follows, for a $\delta$-cluster $C\subseteq P$ we write
$\ell(C)=\sum_{e\in C}\ell(e)$. Define $S=\{C\in M:\
\ell(C)<\tau\delta d\}$. For every $C\in S$, since $C$ is a
maximal $\delta$-cluster, there is an edge $e_C\in P\setminus C$
which is incident with an edge in $C$, such that $\ell(e_C)>\delta
d\ge \frac{\ell(C)}{\tau}$. Note that for every edge $e\in P$,
$|\{C\in S:\ e_C=e\}|\le 2$. Now,
$$
\sum_{C\in S}\frac{\ell(C)}{\tau}\le\sum_{C\in S}\ell(e_C)\le 2\sum_{\substack{e\in P\\
\ell(e)> \delta d}}\ell(e)\le 2\left(d-\sum_{\substack{e\in P\\
\ell(e)\le \delta d}}\ell(e)\right)\le
2\left(1-\frac12-\alpha\right)d=\left(1-2\alpha\right)d.
$$
Using the fact that the path $P$ is
$\left(\frac12+\alpha,\delta\right)$-weak, we see that
$$
\left(\frac12+\alpha\right)d\le \sum_{\substack{e\in P\\
\ell(e)\le\delta d}}\ell(e)=\sum_{\substack{C\in M\\ \ell(C)\ge
\tau\delta d}}\ell(C)+\sum_{C\in S} \ell(C)\le
\sum_{\substack{C\in M\\ \ell(C)\ge \tau\delta
d}}\ell(C)+\left(1-2\alpha\right)\tau d.
$$
Recalling that $\tau=\frac{1}{2-4\alpha}$, we see that $
\sum_{\substack{C\in M\\ \ell(C)\ge \tau\delta d}}\ell(C)\ge
\alpha d, $ as required.

The final assertion of Lemma~\ref{lem:cluster} is simply the fact
that for any weighted path $P=(u_1,\ldots,u_m)$ such that for each
$j\in \{1,\ldots,m-1\}$ we have $\ell(u_j,u_{j+1})\le a$, but
$\sum_{j=1}^{m-1}\ell(u_j,u_{j+1})\ge A$, there are indices
$1=p_1<p_2<\cdots<p_k=m$ such that for all $j\in \{1,\ldots,k-1\}$
we have $\sum_{i=p_j}^{p_{j+1}}\ell(u_i,u_{i+1})\in [A,2A+a]$.
Indeed, let $p_2>p_1$ be the first index such that $
\sum_{i=p_1}^{p_{2}}\ell(u_i,u_{i+1})\ge A$. Then
$\sum_{i=p_1}^{p_{2}}\ell(u_i,u_{i+1})\le A+a$. Continuing
inductively as long as the length of the remaining path is at
least $A$ we find $1=p_1<p_2<\cdots<p_k$ such that for $j\in
\{1,\ldots,k-1\}$ we have
$\sum_{i=p_j}^{p_{j+1}}\ell(u_i,u_{i+1})\in [A,A+a]$, and
$\sum_{i=p_k}^{m}\ell(u_i,u_{i+1})<A$. The required result follows
by replacing $p_k$ with $m$, which increases the length of the
final segment by at most $A$.
\end{proof}

In order to proceed we need to generalize the notions of $\e$-good
and $\delta$-strong colorings. A coloring $\chi:E\to \Z$ will be
called $(\e,\delta)$-strong if it is monotone, and for every $u,v
\in V$
$$
\sum_{k\in \Z} \ell^\chi_k(u,v)\cdot {\bf
1}_{\{\ell^\chi_k(u,v)\ge \delta d_T(u,v)\}}\ge \e d_T(u,v).
$$
Note that we can always assume that $\e\ge \delta$. Using the
terminology of Section~\ref{section:coloring}, an $\e$-good
coloring is the same as an $(\e,\e)$-strong coloring, and a
$\delta$-strong coloring is the same as a $(\frac12,\delta)$-strong
coloring. Thus the following lemma is a generalization of
Lemma~\ref{lem:relation}.

\begin{lemma}\label{lem:epsilon delta} Fix $\e\in (0,\frac12]$ and $\delta\in (0,\e)$. Then any $(\e,\delta)$-strong
coloring is is also a
$\left(\frac{\delta}{4\e}\right)^{3/\e}$-strong coloring.
\end{lemma}

\begin{proof}  The proof is a slight modification of the proof of
Lemma~\ref{lem:relation}. Let $\chi:E\to \mathbb Z$ be an
$(\e,\delta)$-strong coloring, and denote $
\theta=\frac{\e}{2\log\left(\frac{4\e}{\delta}\right)}$. We shall
show that for every $\alpha\in (0,1]$ and $u,v\in V$, the total
length of the monochromatic segments of length at least $\alpha
d_T(u,v)$ on the path $P(u,v)$ satisfies
\begin{eqnarray}\label{eq:geometric-general}
\sum_{k\in \Z} \ell^\chi_k(u,v)\cdot {\bf
1}_{\{\ell^\chi_k(u,v)\ge \alpha d_T(u,v)\}}\ge
\left(1-\left(\frac{\alpha}{\delta}\right)^{\theta}\right)d_T(u,v).
\end{eqnarray}

There are points $a_1,b_1,a_2,b_2,\ldots,a_m,b_m\in V$, ordered
consecutively (from $u$ to $v$) on the path $P(u,v)$, such that
the color classes of length at least $\delta d_T(u,v)$ on the path
$P(u,v)$ are precisely the intervals $\{[a_j,b_j]\}_{j=1}^m$.
Denote for the sake of simplicity $b_0=u$ and $a_{m+1}=v$, and
define $\beta>0$ by $\beta d_T(u,v)=\sum_{j=1}^m d_T(a_j,b_j)$.
Since the coloring is $(\e,\delta)$-strong, we know that $\beta\ge
\e$. By the definition of $\beta$ we are also assured that $m\le
\beta/\delta$. If $\alpha>\delta$ then
inequality~\eqref{eq:geometric-general} holds vacuously, so we
assume that $\alpha\le \delta$. Arguing inductively as in the
proof of Lemma~\ref{lem:relation} we see that
\begin{eqnarray}
&&\nonumber\!\!\!\!\!\!\!\!\!\!\!\!\!\sum_{k\in \Z}
\ell^\chi_k(u,v)\cdot {\bf 1}_{\{\ell^\chi_k(u,v)\ge \alpha
d_T(u,v)\}}\ge \sum_{j=1}^m d_T(a_j,b_j)+
\sum_{j=0}^m\left(1-\left(\frac{\alpha d_T(u,v)}{\delta
d_T(b_j,a_{j+1})}
\right)^{\theta}\right)d_T(b_j,a_{j+1})\\
&=& \nonumber d_T(u,v)-\sum_{j=0}^m\left(\frac{\alpha }{\delta }
\right)^{\theta}\left[d_T(u,v)\right]^\theta\cdot\left[
d_T(b_j,a_{j+1})\right]^{1-\theta}\\ \label{eq:jensen}&\ge&
d_T(u,v)-(m+1)\left(\frac{\alpha }{\delta }
\right)^{\theta}\left[d_T(u,v)\right]^\theta\left(\frac{1}{m+1}\sum_{j=0}^m
d_T(b_j,a_{j+1})\right)^{1-\theta}\\
&=& \nonumber d_T(u,v)-(m+1)\left(\frac{\alpha }{\delta }
\right)^{\theta}\left[d_T(u,v)\right]^\theta\left(\frac{(1-\beta)d_T(u,v)}{m+1}\right)^{1-\theta}\\
&=&\nonumber\left(1-\left(\frac{\alpha}{\delta}\right)^\theta(m+1)^\theta(1-\beta)^{1-\theta}\right)d_T(u,v)\\
&\ge& \label{eq:upper m}
\left(1-\left(\frac{\alpha}{\delta}\right)^\theta\left(\frac{\beta}{\delta}+1\right)^\theta(1-\beta)^{1-\theta}\right)d_T(u,v)\\
&\ge&  \nonumber
\left(1-\left(\frac{\alpha}{\delta}\right)^\theta\left(\frac{2}{\delta}\right)^\theta\beta^\theta(1-\beta)^{1-\theta}\right)d_T(u,v)
\\ \label{eq:monotonicity}
&\ge&\left(1-\left(\frac{\alpha}{\delta}\right)^\theta\left(\frac{2}{\delta}\right)^\theta\e^\theta(1-\e)^{1-\theta}\right)d_T(u,v)\\
&\ge & \label{eq:numerical theta}
\left(1-\left(\frac{\alpha}{\delta}\right)^\theta\right)d_T(u,v),
\end{eqnarray}
where in~\eqref{eq:jensen} we used the concavity of the function
$t\mapsto t^{1-\theta}$, in~\eqref{eq:upper m} we used the fact
that $m\le \beta/\delta$, in~\eqref{eq:monotonicity} we used the
fact that the function $s\mapsto s^\theta(1-s)^{1-\theta}$ is
decreasing on $[\theta,1]$ and that $\beta\ge \e\ge \theta$ (which
follows from the definition of $\theta$ and the fact that $\e\ge
\delta$), and in~\eqref{eq:numerical theta} we used the elementary
inequality
$\left(\frac{2}{\delta}\right)^\theta\e^\theta(1-\e)^{1-\theta}\le
1$, which is equivalent to $\theta\le
\frac{\log(1-\e)}{\log[(1-\e)\delta/(2\e)]}$, and this follows
from the definition of $\theta$ since $\e\le \frac12$.
\end{proof}

 Recall that for $R>0$ a subset $N$ of a metric space $X$ is an
 $R$-net if for every distinct $x,y\in N$ we have $d(x,y)\ge R$
 and for every $z\in X$ there is $x\in N$ with $d(x,z)<R$. In what follows we shall use the following
 variant of this notion.

\begin{definition}\label{def:upward} Let $T=(V,E)$ be a tree rooted at $r$ with
edge weights $\ell:E\to (0,\infty)$. For $R>0$ we shall call a set
$N\subseteq V$ an upward $R$-net of $T$ if for every $x,y\in N$
such that $x$ is an ancestor of $y$ we have $d_T(x,y)\ge R$ and
for every $v\in V$ there is $x\in P(v)\cap N$ such that $d_T(v,x)<
R$. In other words, $N$ is an upward $R$-net of $T$ if and only if
for every $v\in V$, $N\cap P(v)$ is an $R$-net in $P(v)$.
\end{definition}

Observe that an upward $R$-net in $T$ need not be an $R$-net in
$T$. However, the following easy lemma shows that upward $R$-nets
always exist.

\begin{lemma}\label{lem:upward exists} $T$ admits an upward $R$-net
for every $R>0$.
\end{lemma}

\begin{proof} The proof is an easy induction on $V$. For $|V|=1$
the result is trivial. Assume that $|V|>1$ and let $v\in V$ be a
leaf of $T$. Let $u\in V$ be the father of $v$. By the inductive
hypothesis the tree $T'=(V\setminus\{v\},E\setminus\{(u,v)\})$
admits an upward $R$-net $N'$. Thus there exists $x\in N\cap P(u)$
such that $d_{T'}(x,u)=d_T(x,u)< R$. If $\ell(u,v)\ge R-d_T(x,u)$
define $N=N'\cup\{v\}$, which is clearly an upward $R$-net in $T$.
Otherwise $d_T(v,x)<R$, and since $x\in P(v)$, it follows that
$N'$ is also an upward $R$-net in $T$.
\end{proof}

\subsection{Construction of a special coloring and the proof of
Theorem~\ref{thm:Markov main}}

Our basic strategy is similar to the proof of
Theorem~\ref{thm:upperB}. To emphasize the similarities between
the two proofs, we will use the same notation for the weight
functions and the coloring that we construct (this will not cause
any confusion since Section~\ref{sec:binary} can be read
independently of the present section). As in the proof of
Theorem~\ref{thm:upperB} we will define a weight function $\mu_j$
on subtrees of $T$, and a ``scale selector" $g:V\to \mathbb Z\cup
\{\infty\}$, which will be used to construct a coloring $\chi$ of
$T$. The fact that $\chi$ is not $\delta$-strong will be used to
find an appropriate copy of a weak prototype in $T$.

We begin with some notation. Let $Q$ be a (weighted) path with
initial vertex $x$ and final vertex $y$, and let $F$ be an
arbitrary tree with root $y$ (but otherwise disjoint from $Q$).
For $\e,\delta\in (0,1)$ and $L>0$ we define
$\rho(\e,\delta,L;Q,F)$ to be the least minimum distance from the
root to a leaf in a {\em subtree} $F'\subseteq F$ which satisfies
the following three conditions:

\begin{enumerate}
\item Every non-leaf vertex of $F'$ has exactly one or two children.

\item Let $P$ be a root-leaf path in $F'$, and let $\widetilde P$
be the vertices on $P$ which are either one of the endpoints of
$P$ or have $2$ children in $F'$. Then the path $Q\cup \widetilde
P$ is $(\e,\delta)$-weak.

\item Every path from $x$ to a leaf of $F'$ has length at most
$3L$.
\end{enumerate}

Next, we construct a monotone coloring $\chi:E\to \Z$ and a
``scale selector" $g:V\to \Z\cup\{\infty\}$ in a similar way to
what was done in Section~\ref{sec:binary}. Along the way we will
also construct weight functions $\{\mu_s\}_{s\in \Z}$ on subtrees
of $T$. As in Section~\ref{sec:binary} we start by setting
$g(r)=\infty$ and we assume inductively that the construction is
done so that whenever $v\in V$ is such that $g(v)$ is defined, if
$u$ is a vertex on the path $P(v)$ then $g(u)$ has already been
defined, and for every edge $e\in E$ incident to $v$, $\chi(e)$
has been defined.

For every $t\in \Z$ let $N_t$ be an upward $4^t$-net of $T$.
Since $N_t$ is an
upward $4^t$-net, for all $w\in V$ we are assured that $ N_t\cap
P(w)\cap B_T(w,4^t)\neq \emptyset$.  We define $\lambda_t(w)$ to be the
point in $N_t\cap P(w)\cap B_T(w,2\cdot 4^t)$ which is furthest
away from $w$.  Now let $t(s)\in \mathbb Z$ be such that
$$
240\cdot \delta^{-\frac{1}{2880}}\cdot 4^s\le 4^{t(s)}< 960\cdot
\delta^{-\frac{1}{2880}}\cdot 4^s.
$$

Take $v\in V$ which is the vertex closest to the root $r$ for
which $g(v)$ hasn't yet been defined, and as in
Section~\ref{sec:binary} we set
\begin{eqnarray}\label{eq:def g}
g(v)=\max\left\{j\in \mathbb Z:\ \forall\ u\in \beta_\chi(v),\
d_T(u,v)\ge 4^{\min \{g(u),j\}}\right\}.
\end{eqnarray}
Recall that $\beta_\chi(v)$ denotes the set of breakpoints of $\chi$
along the path $P(v)$, and that by the inductive hypothesis the path
$P(v)$ has been entirely colored. Let $F$ be a subtree of $T$ rooted
at $v$. We shall now define $\mu_s(F)$. To this end define a subset
of the path $P(v)$ by
\begin{eqnarray}\label{eq:define break path}
Q_s(v)=\{\lambda_{t(s)}(v)\}\bigcup \left\{w\in \beta_\chi(v):\
g(w)=g(v)\ \mathrm{and}\ \lambda_{t(s)}(w)=\lambda_{t(s)}(v) \right\}.
\end{eqnarray}
With this notation we can define
\begin{eqnarray}\label{eq:new weight}
\mu_s(F)=\rho\left(\frac{1}{2500},\delta^{\frac{1}{2880}},4^{t(s)};Q_s(v),F\right).
\end{eqnarray}
In~\eqref{eq:new weight} we extended the definition of $\mu_s$
to all subtrees of $T$ rooted at $v$. We next choose one of the
children of $v$, $w\in \mathscr{C}(v)$, for which
$$
\mu_{g(v)} (F_w)=\max_{z\in \mathscr{C}(v)} \mu_{g(v)}(F_z).
$$
Observe that $\mu_{g(v)}(F_z)$ is defined for all the children of
$v$, since $F_z$ is a subtree of $T$ rooted at $v$ (it is the
subtree rooted at $z\in \mathscr C(v)$ together with the incoming
edge $\{v,z\}$). Letting $u$ be the father of $v$ on the path
$P(v)$, we set $\chi(v,w)=\chi(u,v)$, and we assign arbitrary new
(i.e. which haven't been used before) distinct colors to each of
the edges $\{(v,z)\}_{z\in \mathscr C(v)\setminus \{w\}}$.

This construction yields a monotone coloring $\chi$, a function
$g:V\to \Z\cup\{\infty\}$, and weight function $\{\mu_s\}_{s\in
\Z}$ defined on subtrees of $T$. In particular, we note here that
Claim~\ref{claim:g} still holds true, since its proof only used
the fact that $g$ was defined as in~\eqref{eq:def g}, and this
formula is identical to the one used in Section~\ref{sec:binary}.
The following lemma contains the crucial properties of the
coloring $\chi$.

\remove{

Denoting the root of $T$ by $r$, recall that for $v\in V$ the path
joining $r$ and $v$ is denoted $P(v)$. Define $S_j(v)\subseteq V$ by
$$
S_j(v)= \left(P(v)\cap N_j\cap \left\{u\in V:\ d_T(u,v)\ge
2^j\right\} \right)\cup\{r\}.
$$
Let $\kappa_j(v)$ be a the vertex of $S_j(v)$ which is closest to
$v$, i.e. the first point of $N_j$ that was encountered after having
gone ``up the tree" for a distance at least $2^j$. We also let
$\lambda_j(v)$ be the vertex of $B_T\left(v,2^{j+1}\right)\cap
P(\kappa_j(v),v)$ which is closest to $\kappa_j(v)$. Finally, we
define a weight function $\mu_j$ on subtrees of $T$ as follows. Let
$F$ be a subtree of $T$ rooted at $r_F$. Let $i\in \mathbb Z$ be
such that $\frac{2^j}{100\delta}\le 2^i<\frac{2^{j+1}}{100\delta}$.
Denoting by $Q_F$ the path from $\kappa_i(r_F)$ to $r_F$ we define
\begin{eqnarray}\label{eq:new weight}
\mu_j(F)=\rho\left(\frac{1}{10000},100\delta,2^i;Q_j(F),F\right).
\end{eqnarray}
We can now repeat the construction that was carried out in the proof
of Theorem~\ref{thm:upperB} using the new weight functions $\mu_j$
in~\eqref{eq:new weight}. This procedure produces a monotone
coloring $\chi:E\to \Z$ and a function $g:V\to \Z\cup\{\infty\}$.
The following lemma contains the crucial properties of the coloring
$\chi$. }

\begin{lemma}\label{lem:crucial coloring} Assume that the above
coloring $\chi$ is not $\delta$-strong. Then there exists a
sequence of vertices $Q=(x,w_1,\ldots,w_N)$, ordered down the
tree, and a number $L>0$, such that if we define $s,t\in \Z$ by $
4^{s-1}<\frac{1}{240}\delta^{\frac{1}{2880}}L\le 4^s$ and
$240\delta^{-\frac{1}{2880}}4^s\le 4^t<
960\delta^{-\frac{1}{2880}}4^s$, then the path metric induced by
$T$ on $Q$ has the following properties:
\begin{enumerate}
\item For every $j\in \{1,\ldots,N\}$ the vertex $w_j$ is a
breakpoint of $\chi$.
\item For every $j\in \{1,\ldots,N\}$ we have $g(w_j)=s$ and
$\lambda_t(w_j)=x$.
\item The path $Q$ is
$\left(\frac{1}{2500},\delta^{\frac{1}{2880}}\right)$-weak.
\item The length of $Q$ satisfies $\ell(Q)=d_T(x,w_N)\in
\left[\frac{L}{850},3L\right]$.
\end{enumerate}
\end{lemma}

Before passing to the proof of Lemma~\ref{lem:crucial coloring} we
show how it can be used to complete the proof of
Theorem~\ref{thm:Markov main}.

\begin{proof}[Proof of Theorem~\ref{thm:Markov main}] With
Lemma~\ref{lem:crucial coloring} at hand, the proof of
Theorem~\ref{thm:Markov main} is similar to the final step of
the proof of Theorem~\ref{thm:upperB}. Assume that
$\delta>\delta^*(T)$. Let $Q=(x,w_1,\ldots,w_N)$ be the path
constructed in Lemma~\ref{lem:crucial coloring}, and we shall also
use the same $s,t,L$ obtained there. Observe that using the
notation in~\eqref{eq:define break path} we may assume that
$Q=Q_s(w_N)$. Indeed, by adding to $Q$ any additional breakpoint
$w$ of $\chi$ along $P(w_N)$ with $g(w)=g(w_N)=s$ and
$\lambda_t(w)=\lambda_t(w_N)=x$ we do not change the conclusion
Lemma~\ref{lem:crucial coloring}.

For $i\in \{1,\ldots,N-1\}$ let $z_i$ be the child of $w_i$ along
the path $P(x,w_N)$, and denote for the sake of simplicity
$z_N=w_N$. We shall prove by induction on $i\ge 0$ that the
subtree of $T$ rooted at $z_{N-i}$ (i.e. the tree $T_{z_{N-i}}$)
has a further sub-tree $W_i$ satisfying the following properties.
\begin{enumerate}
\item Every path from $x$ to a leaf of $W_i$ has length at most
$3\cdot 4^t$.
\item Every non-leaf vertex of $W_i$ has exactly one or two children.
\item Let $P$ be a path from $z_{N-i}$ to a leaf of $W_i$, and let $\widetilde P$
be the vertices on $P$ which are either one of the endpoints of
$P$ or have $2$ children in $W_i$. Then the path
$(x,w_1,\ldots,w_{N-i})\cup \widetilde P$ is
$\left(\frac{1}{2500},\delta^{\frac{1}{2880}}\right)$-weak.
\item Every root-leaf path in $W_i$ has length at least
$d_T(z_{N-i},w_N)$.
\end{enumerate}

For $i=0$ we just take $W_0$ to be the singleton $w_N$, and the
fact that  the required properties are satisfied is asserted in
Lemma~\ref{lem:crucial coloring}. Similarly, for $i=1$ we let
$W_1$ be the tree consisting of the single edge $(z_{N-1},w_N)$
which satisfies the required properties due to
Lemma~\ref{lem:crucial coloring}. Assuming the existence of $W_i$
we proceed inductively as follows. Since $w_{N-i}$ is a breakpoint
of $\chi$, the construction of $\chi$ in the proof of
Theorem~\ref{thm:upperB}, and the fact that $g(w_{N-i})=s$,
implies that there is a child $z_{N-i}'$ of $w_{N-i}$ for which
$\mu_s(F_{z'_{N-i}})>\mu_s(F_{z_{N-i}})$ (recall that for $u\in V$
the tree $F_u$ is the subtree rooted at $u$ plus the edge joining
$u$ and its parent in $T$). Now, since $Q_s(w_N)=Q$ we also know
that (since $\lambda_t(w_{N-i})=x$)
$Q_s(w_{N-i})=\left\{x,w_1,\ldots,w_{N-i}\right\}$. Thus by the
definition on $\mu_s$ in~\eqref{eq:new weight}
$$
\mu_s(F_{z_{N-i}})=\rho\left(\frac{1}{2500},\delta^{\frac{1}{2880}},4^t;Q_s(w_{N-i}),F_{z_{N-i}}\right)=
\rho\left(\frac{1}{2500},\delta^{\frac{1}{2880}},4^t;\left\{x,w_1,\ldots,w_{N-i}\right\},F_{z_{N-i}}\right).
$$
But, $W_i$ is a subtree of $F_{z_{N-i}}$ in which every non-leaf
vertex has two children, for every path from $x$ to a leaf of
$W_i$ the path metric induced by $T$ on the vertices which are
either $x$, or one of the $w_j$, or a leaf in $W_i$, or have  2
children in $W_i$, is
$\left(\frac{1}{2500},\delta^{\frac{1}{2880}}\right)$-weak, every
path from $x$ to a leaf of $W_i$ has length at most $3L\le 3\cdot
4^t$, and the minimal distance from a root to a leaf of $W_i$ is
at least $d_T(z_{N-i},w_N)$. Thus the definition of $\rho$ implies
that $\mu_s(F_{z_{N-i}})\ge d_T(z_{N-i},w_N)$. It follows that
$\mu_s(F_{z'_{N-i}})> d_T(z_{N-i},w_N)$, implying the existence of
a subtree $W_i'$ of $F_{z'_{N-i}}$, which has the same properties
as those stated above for $W_i$. Joining these two subtrees at
$w_{N-i}$, and adding an edge from $z_{N-i+1}$ to $w_{N-i}$ we
obtain a subtree $W_{N-i+1}$ rooted at $z_{N-i+1}$ with the
desired properties. We recommend that the reader will follow the
above construction using a drawing analogous to
Figure~\ref{figure:gluing}.

The tree $T'$ obtained by joining the edges $(x,w_1), (w_1,z_1)$
to $W_{N-1}$ is a subtree of $T$ which is a
$\left(\frac{1}{2500},\delta^{\frac{1}{2880}}\right)$-weak
prototype with height ratio at most $\frac{3\cdot
4^t}{d(x,w_N)}\le \frac{8L}{L/850}=6800$. As explained in the
discussion following Theorem~\ref{thm:Markov main}, this completes
the proof.
\end{proof}

Thus, all that remains is to prove Lemma~\ref{lem:crucial coloring}.

\begin{proof}[Proof of Lemma~\ref{lem:crucial coloring}]
Since we are assuming that the coloring $\chi$ is not
$\delta$-strong, Lemma~\ref{lem:epsilon delta} implies that $\chi$
is also not
$\left(\frac{1}{960},\frac{1}{240}\delta^{\frac{1}{2880}}\right)$-strong.
Thus there exist two vertices $u,v\in V$ such that more than a
$\frac{959}{960}$-fraction of the length of the path joining $u$
and $v$ is covered by color classes of length less than
$\frac{1}{240}\delta^{\frac{1}{2880}}\cdot d_T(u,v)$. Let
$(b_1,\ldots,b_m)$ be the breakpoints of the coloring $\chi$ along
the path $P(u,v)$, ordered from $u$ to $v$. We also denote $b_0=u$
and $b_{m+1}=v$. Thus
$$
\sum_{j=1}^{m+1} d_T(b_{j-1},b_j)\cdot {\bf
1}_{\left\{d_T(b_{j-1},b_j)<
\frac{1}{240}\delta^{\frac{1}{2880}}\cdot d_T(u,v)\right\}}\ge
\frac{959}{960}
d_T(u,v)=\left(\frac12+\frac{479}{960}\right)d_T(u,v).
$$
Lemma~\ref{lem:cluster} (with $\alpha=\frac{479}{960}$ and
$\tau=\frac{1}{2-4\alpha}=240$) implies that there exists a sequence
of indices $0\le p_1<q_1\le p_2< q_2\le \cdots\le p_{k-1}<q_{k-1}\le
p_k<q_k\le m+1$ such that for every $1\le i\le k$ we have
$d_T(b_{p_i},b_{q_i})\in\left[
\delta^{\frac{1}{2880}}d_T(u,v),3\delta^{\frac{1}{2880}}d_T(u,v)\right]$
and every $p_i<j\le q_i$ satisfies $d_T(b_{j-1},b_j)\le
\frac{1}{240}\delta^{\frac{1}{2880}} d_T(u,v)$. Moreover, the total
length of these ``long
$\frac{1}{240}\delta^{\frac{1}{2880}}$-clusters" is
\begin{eqnarray}\label{eq:sum pi} \sum_{i=1}^k
d_T\left(b_{p_i},b_{q_i}\right)\ge \frac{479}{960} d_T(u,v).
\end{eqnarray}
It follows in particular from~\eqref{eq:sum pi} that $k\ge
\frac{479}{3\cdot 960}\cdot \delta^{-\frac{1}{2880}}\ge 20$ (since
$\delta< (140)^{-2880}$).

Denote $L=d_T(u,v)$ and recall that $s\in \Z$ is defined by $
4^{s-1}<\frac{1}{240}\delta^{\frac{1}{2880}}L\le 4^s$. Fix $1\le
i\le k$ and apply Claim~\ref{claim:g} to the path
$P(b_{p_i},b_{q_i})$ with $c=2$ (which we are allowed to do by the
definition of $s$). It follows that there exist at least two
indices $p_i\le j_1(i)<j_2(i)\le q_i$ such that
$g(b_{j_1(i)})=g(b_{j_2(i)})=s$ and $ 9\cdot 4^s\le
d_T(b_{j_1(i)},b_{j_2(i)})\le 18\cdot 4^s$.

Now $t\in \Z$ is given by by $240\delta^{-\frac{1}{2880}}4^s\le
4^t< 960\delta^{-\frac{1}{2880}}4^s$. Note that by the definition
of $s$ this implies that $L\le 4^t<16L$. For each point $w\in
\{b_{j_1(1)},b_{j_2(1)},\ldots,b_{j_1(k)},b_{j_1(k)} \}$ the
vertex $\lambda_s(w)$ is in $B_T(w,2\cdot 4^t)\cap N_t\cap
P(w)\subseteq B_T(v,2\cdot 4^t+L)\cap N_t\cap P(v)\subseteq
B_T(v,3\cdot 4^t)\cap N_t\cap P(v)$. Since $N_t\cap P(v)$ is
$4^t$-separated, it follows that there are at most $4$ possible
vertices which could equal $\lambda_t(w)$. Thus there is a vertex
$x\in V$ and a subinterval $J\subseteq \{1,\ldots,k\}$ of size at
least $\frac{k}{4}-1\ge \frac{k}{5}$ (since $k\ge 20$) such that
for all $i\in J$ we have
$\lambda_t(b_{j_1(i)})=\lambda_t(b_{j_2(i)})=x$. Note that since
$x=\lambda_t(w)$ for some $w\in
\{b_{j_1(1)},b_{j_2(1)},\ldots,b_{j_1(k)},b_{j_1(k)} \}$, we know
that $x$ is the point in $N_t\cap P(w)\cap B_T(w,2\cdot 4^t)$
which is furthest from $w$. Since $N_t$ is an upward $4^t$-net,
there is a point $y\in N_t\cap P(u)\cap B_T(u,4^t)$. So, using
$d_T(w,u)\le L\le 4^t$, we see that $y\in N_t\cap P(w)\cap
B_T(w,2\cdot 4^t)$. Thus $x\in P(y)\subseteq P(u)$, i.e. $x$ is
closer to the root than $u$.

Consider the path metric induced on the vertices
$Q=\{x\}\cup\{b_{j_1(i)}\}_{i\in J}\cup \{b_{j_2(i)}\}_{i\in J}$.
For simplicity of notation we enumerate it down the tree by
$Q=(x,w_1,\ldots,w_N)$. We bound the length of $Q$ as follows.
First, $ \ell(Q)=d_T(x,w_N)=d_T(x,w_1)+d_T(w_1,w_N)\le 2\cdot
4^t+L\le 3 L$. On the other hand, using~\eqref{eq:sum pi} we see
that
\begin{multline*}
\ell(Q)\ge \sum_{i\in J} d_T(b_{j_1(i)},b_{j_2(i)})\ge \sum_{i\in J}
9\cdot 4^s\ge  \frac{k}{5}\cdot 9\cdot 4^s\\ \ge
\frac{9}{5}\sum_{i=1}^k \frac{1}{240}\delta^{\frac{1}{2880}}L\ge
\frac{3}{400}\sum_{i=1}^k \frac13 d_T(b_{p_i},b_{q_i})\ge
\frac{479}{400\cdot 960}L \ge \frac{\ell(Q)}{2500}.
\end{multline*}
This also shows that the path $Q$ is
$\left(\frac{1}{2500},\delta^{\frac{1}{2880}}\right)$-weak,
completing the proof of Lemma~\ref{lem:crucial coloring}. \remove{
\begin{figure}\label{figure:weak prototype}
 \centerline{\hbox{
        \psfig{figure=special-path.eps,height=2.5in}
        }}
        \caption{}\label{figure:special-path}
\end{figure}
}
\end{proof}


\def\cprime{$'$}

\end{document}